\documentclass{amsart}

\usepackage[T1]{fontenc}
\usepackage[latin1]{inputenc}
\usepackage{amsmath,amsthm,amstext,amsfonts,amssymb,amscd,amsxtra}
\usepackage{bm}
\usepackage{enumerate}
\usepackage{eucal}
\usepackage{mathrsfs}
\usepackage{stmaryrd}
\usepackage[all]{xy}

\allowdisplaybreaks

\theoremstyle{plain}
\newtheorem{theorem}{Theorem}[section]

\newtheorem*{theoremA}{Theorem~A}
\newtheorem*{theoremB}{Theorem~B}
\newtheorem*{corollaryC}{Corollary~C}
\newtheorem*{corollaryD}{Corollary~D}

\newtheorem{lemma}[theorem]{Lemma}

\newtheorem{proposition}[theorem]{Proposition}
\newtheorem{proposition-definition}[theorem]{Proposition-Definition}
\newtheorem{corollary}[theorem]{Corollary}
\newtheorem{claim}[theorem]{Claim}

\theoremstyle{definition}

\theoremstyle{remark}
\newtheorem{remark}[theorem]{Remark}

\newcommand{\NN}{\mathbb{N}}
\newcommand{\ZZ}{\mathbb{Z}}
\newcommand{\QQ}{\mathbb{Q}}
\newcommand{\RR}{\mathbb{R}}
\newcommand{\CC}{\mathbb{C}}

\newcommand{\KK}{\mathbb{K}}

\newcommand{\PP}{\mathbb{P}}

\makeatletter
 
 \@addtoreset{equation}{section}
\makeatother

\makeatletter
\newcommand{\esssup}{\mathop{\operator@font ess.sup}\displaylimits}
\newcommand{\essinf}{\mathop{\operator@font ess.inf}\displaylimits}
\makeatother

\makeatletter 
\let\@@pmod\pmod 
\DeclareRobustCommand{\pmod}{\@ifstar\@pmods\@@pmod} 
\def\@pmods#1{\mkern4mu({\operator@font mod}\mkern 6mu#1)} 
\makeatother

\newcommand{\Image}{\mathop{\mathrm{Image}}\nolimits}
\renewcommand{\leq}{\leqslant}
\renewcommand{\geq}{\geqslant}

\DeclareMathOperator{\Hz}{H^0}

\DeclareMathOperator{\Sym}{Sym}
\DeclareMathOperator{\Spec}{Spec}

\DeclareMathOperator{\Supp}{Supp}

\DeclareMathOperator{\Rat}{Rat}

\DeclareMathOperator{\Bs}{Bs}

\DeclareMathOperator{\Div}{Div}

\DeclareMathOperator{\mult}{mult}

\DeclareMathOperator{\aBigCone}{\widehat{Big}}
\DeclareMathOperator{\aNef}{\widehat{Nef}}

\DeclareMathOperator{\aDiv}{\widehat{Div}}

\DeclareMathOperator{\adeg}{\widehat{deg}}

\DeclareMathOperator{\vol}{vol}
\DeclareMathOperator{\avol}{\widehat{vol}}

\DeclareMathOperator{\rk}{rk}

\DeclareMathOperator{\sgn}{sgn}

\def\Largesymbol#1{\mbox{\strut\rlap{\smash{\Large$#1$}}\quad}}

\title{On the concavity of the arithmetic volumes}
\author{Hideaki Ikoma}
\date{August 14, 2014.}
\thanks{This research is supported by Research Fellow of Japan Society for the Promotion of Science.}
\address{Graduate School of Mathematical Sciences, The University of Tokyo, Tokyo, 153-8914, Japan}
\email{ikoma@ms.u-tokyo.ac.jp}
\subjclass{Primary 14G40; Secondary 11G50, 37P30}
\begin{document}

\begin{abstract}
In this paper, we study the differentiability of the arithmetic volumes along arithmetic $\RR$-divisors, and give some equality conditions for the Brunn-Minkowski inequality for arithmetic volumes over the cone of nef and big arithmetic $\RR$-divisors.
\end{abstract}

\maketitle
\tableofcontents

\section{Introduction}

Let $X$ be a normal projective arithmetic variety of dimension $d+1$, and denote the rational function field of $X$ by $\Rat(X)$.
Following Moriwaki \cite{Moriwaki12a}, we consider an arithmetic $\RR$-divisor $\overline{D}$ on $X$ (see \S 2 for definitions).
In this paper, we suppose that all arithmetic $\RR$-divisors are $\RR$-Cartier and of $C^0$-type.
The arithmetic volume of $\overline{D}$ is defined as
\[
 \avol(\overline{D}):=\limsup_{m\to\infty}\frac{\log\sharp\{s\in\Hz(X,mD)\,|\,\|s\|_{\sup}^{m\overline{D}}\leq 1\}}{m^{d+1}/(d+1)!},
\]
where $\|\cdot\|_{\sup}^{m\overline{D}}$ is the supremum norm on $\Hz(X,mD)\otimes_{\ZZ}\RR$ defined by the Green function of $m\overline{D}$.
In \cite{Chen11}, H. Chen proved that the function $\avol$ is differentiable at every big arithmetic divisor along the directions defined by arbitrary arithmetic divisors.
In this paper, we generalize this result to arithmetic $\RR$-divisors: that is, we prove that, for a big arithmetic $\RR$-divisor $\overline{D}$ and for an arithmetic $\RR$-divisor $\overline{E}$, the function $\RR\ni t\mapsto\avol(\overline{D}+t\overline{E})\in\RR$ is differentiable and
\[
 \lim_{t\to 0}\frac{\avol(\overline{D}+t\overline{E})-\avol(\overline{D})}{t}=(d+1)\langle\overline{D}^{\cdot d}\rangle\overline{E},
\]
where $\langle\overline{D}^{\cdot d}\rangle\overline{E}$ is the arithmetic positive intersection number defined in \S\ref{sec:aposint} (Theorem~\ref{thm:diffavol}).
A merit of such generalization is that we can obtain the following arithmetic version of the Discant inequality, which was proved by Discant \cite{Discant} in the context of convex geometry and by Boucksom-Favre-Jonsson \cite{Bou_Fav_Mat06} in the context of algebraic geometry.

\begin{theoremA}[Theorem~\ref{thm:discant}]
Let $\overline{D}$ and $\overline{P}$ be two big arithmetic $\RR$-divisors.
If $\overline{P}$ is nef, then we have
\[
 0\leq\left(\left(\langle\overline{D}^{\cdot d}\rangle\overline{P}\right)^{\frac{1}{d}}-s\avol(\overline{P})^{\frac{1}{d}}\right)^{d+1}\leq\left(\langle\overline{D}^{\cdot d}\rangle\overline{P}\right)^{1+\frac{1}{d}}-\avol(\overline{D})\avol(\overline{P})^{\frac{1}{d}},
\]
where $s=s(\overline{D},\overline{P}):=\sup\{t\in\RR\,|\,\text{$\overline{D}-t\overline{P}$ is pseudo-effective}\}$.
\end{theoremA}

As was pointed out in \cite{Cutkosky13}, Theorem~A immediately gives explicit bounds for $s(\overline{D},\overline{P})$ (see also \cite[Problem~B]{TeissierBonn}) and a Bonnesen-type inequality in the arithmetic context (Corollary~\ref{cor:Teissier}).
In \cite{Yuan09}, X. Yuan proved that the arithmetic volumes fit in the Brunn-Minkowski-type inequality:
\[
 \avol(\overline{D}+\overline{E})^{\frac{1}{d+1}}\geq\avol(\overline{D})^{\frac{1}{d+1}}+\avol(\overline{E})^{\frac{1}{d+1}}
\]
for all pseudo-effective arithmetic $\RR$-divisors $\overline{D}$ and $\overline{E}$ (the continuity property of the arithmetic volume function is due to Moriwaki \cite{Moriwaki12a}).
A main purpose of this paper is to obtain equality conditions for the Brunn-Minkowski inequality over the cone of nef and big arithmetic $\RR$-divisors.

\begin{theoremB}[Theorem~\ref{thm:strictconcave}]
Let $\overline{D}$ and $\overline{E}$ be two nef and big arithmetic $\RR$-divisors.
Then the following are all equivalent.
\begin{enumerate}
\item[\textup{(1)}] $\avol(\overline{D}+\overline{E})^{\frac{1}{d+1}}=\avol(\overline{D})^{\frac{1}{d+1}}+\avol(\overline{E})^{\frac{1}{d+1}}$.
\item[\textup{(2)}] For any $i$ with $1\leq i\leq d$, we have $\adeg(\overline{D}^{\cdot i}\cdot\overline{E}^{\cdot (d-i+1)})=\avol(\overline{D})^{\frac{i}{d+1}}\cdot\avol(\overline{E})^{\frac{d-i+1}{d+1}}$.
\item[\textup{(3)}] $\adeg(\overline{D}^{\cdot d}\cdot\overline{E})=\avol(\overline{D})^{\frac{d}{d+1}}\cdot\avol(\overline{E})^{\frac{1}{d+1}}$.
\item[\textup{(4)}] There exist $\phi_1,\dots,\phi_l\in\Rat(X)^{\times}$ and $a_1,\dots,a_l\in\RR$ such that
\[
 \frac{\overline{D}}{\avol(\overline{D})^{\frac{1}{d+1}}}-\frac{\overline{E}}{\avol(\overline{E})^{\frac{1}{d+1}}}=a_1\widehat{(\phi_1)}+\dots+a_l\widehat{(\phi_l)}.
\]
\end{enumerate}
\end{theoremB}

To prove Theorem~B, the generalized Dirichlet unit theorem of Moriwaki \cite{Moriwaki10b} plays an essential role (Theorem~\ref{thm:psef}).
As applications, we give some characterizations of the Zariski decompositions over high dimensional arithmetic varieties. 
The following were proved by Moriwaki \cite{Moriwaki12b} when $\dim X$ is two, and used to characterize the Zariski decompositions over arithmetic surfaces in terms of the arithmetic volumes.

\begin{corollaryC}[Corollary~\ref{cor:equalnefbig}]
Let $\overline{P}$ and $\overline{Q}$ be two nef and big arithmetic $\RR$-divisors.
If $\avol(\overline{P})=\avol(\overline{Q})$ and $\overline{Q}-\overline{P}$ is effective, then $\overline{P}=\overline{Q}$.
\end{corollaryC}

\begin{corollaryD}[Corollary~\ref{cor:Zariski}]
Let $\overline{D}$ be a big arithmetic $\RR$-divisor on $X$.
Then there exists at most one decomposition $\overline{D}=\overline{P}+\overline{N}$ such that
\begin{enumerate}
\item[\textup{(1)}] $\overline{P}$ is a nef arithmetic $\RR$-divisor,
\item[\textup{(2)}] $\overline{N}$ is an effective arithmetic $\RR$-divisor, and
\item[\textup{(3)}] $\avol(\overline{P})=\avol(\overline{D})$.
\end{enumerate}
Such a decomposition, if it exists, is called a Zariski decomposition of $\overline{D}$.
\end{corollaryD}

It is known that, if $X$ is a regular projective arithmetic surface, then a Zariski decomposition of a big arithmetic $\RR$-divisor $\overline{D}$ always exists (\cite{Moriwaki12a}) and, if $\dim X$ is bigger than two, then there is no Zariski decomposition of $\overline{D}$ in general even after any blowing up of $X$.

This paper is organized as follows:
in \S 2, we recall some positivity notions for arithmetic $\RR$-divisors and deduce Khovanskii-Teissier-type inequalities from the arithmetic Hodge index theorem (Theorem~\ref{thm:aKT}).
In \S 3, we define the arithmetic positive intersection numbers for arithmetic $\RR$-divisors.
In \S 4, we prove a limit formula expressing the arithmetic positive intersection numbers in terms of asymptotic intersection numbers of moving parts (Proposition~\ref{prop:aposint}).
We can use this as an alternative definition for the arithmetic positive intersection numbers.
In \S 5, we establish the differentiability of the arithmetic volume functions along arithmetic $\RR$-divisors (Theorem~\ref{thm:diffavol}).
The proof is based on the arguments due to Boucksom-Favre-Jonsson \cite{Bou_Fav_Mat06}.
As in \cite{Chen11}, we also apply the results to the problem of equidistribution of rational points (Corollary~\ref{cor:equidistthm}).
In \S 6, we give a numerical characterization of pseudo-effective arithmetic $\RR$-divisors (Theorem~\ref{thm:psef}), which is an arithmetic analogue of the results of Boucksom-Demailly-Paun-Peternell \cite{Boucksom_Demailly_Paun_Peternell}.
Finally, in \S 7, we prove the main results, Theorems~A (Theorem~\ref{thm:discant}) and B (Theorem~\ref{thm:strictconcave}) and Corollaries~C (Corollary~\ref{cor:equalnefbig}) and D (Corollary~\ref{cor:Zariski}).

\section{Arithmetic Khovanskii-Teissier inequalities}\label{sec:arithvol}

Let $X$ be a \emph{projective arithmetic variety}, that is, a reduced irreducible scheme projective and flat over $\Spec(\ZZ)$.
Throughout this paper, we always assume that $X$ is normal.
We denote the dimension of $X$ by $d+1$, and the complex analytic space associated to $X_{\CC}:=X\times_{\Spec(\ZZ)}\Spec(\CC)$ by $X(\CC)$.
We say that $X$ is \emph{generically smooth} if the generic fiber $X_{\QQ}:=X\times_{\Spec(\ZZ)}\Spec(\QQ)$ is smooth.
A \emph{$C^0$-function} on $X$ is a real-valued continuous function on $X(\CC)$ that is invariant under the complex conjugation.
We denote the $\RR$-vector space of all $C^0$-functions on $X$ by $C^0(X)$.
When we consider a $C^{\infty}$-function on $X(\CC)$, we always assume that $X$ is generically smooth.
Let $\KK$ be either $\RR$ or $\QQ$ and let $\mathscr{T}$ be either $C^0$ or $C^{\infty}$.
Let $D$ be a $\KK$-divisor on $X$, which can be written as a sum $D=a_1D_1+\dots+a_lD_l$ with $a_1,\dots,a_l\in\KK$ and effective Cartier divisors $D_1,\dots,D_l$.
A \emph{$D$-Green function of $C^0$-type} (resp.\ \emph{$D$-Green function of $C^{\infty}$-type}) is a continuous function $g_{\overline{D}}:(X\setminus\bigcup_{i=1}^l\Supp(D_i))(\CC)\to\RR$ such that $g_{\overline{D}}$ is invariant under the complex conjugation and that for each $p\in X(\CC)$ there exists an open neighborhood $U\subset X(\CC)$ of $p$ such that the function
\[
 g_{\overline{D}}(x)+\sum_{i=1}^la_i\log|f_i(x)|^2
\]
extends to a $C^0$-function (resp.\ $C^{\infty}$-function) on $U$, where $f_i$ denotes a local defining equation for $D_i$ on $U$.
One can verify that this definition does not depend on the choice of the expression $D=a_1D_1+\dots+a_lD_l$ and the local defining equations $f_1,\dots,f_l$.
We call the pair $\overline{D}:=(D,g_{\overline{D}})$ consisting of a $\KK$-divisor $D$ and a $D$-Green function $g_{\overline{D}}$ of $\mathscr{T}$-type an \emph{arithmetic $\KK$-divisor of $\mathscr{T}$-type} on $X$.
We denote the $\KK$-vector space of all arithmetic $\KK$-divisors on $X$ of $\mathscr{T}$-type by $\aDiv_{\KK}(X;\mathscr{T})$.
Let $x\in X(\overline{\QQ})$ be a rational point, let $K(x)$ be the minimal field of definition for $x$, and let $C_x$ be the normalization of the arithmetic curve $\overline{\{x\}}$.
If $x\in(X\setminus\bigcup_{i=1}^l\Supp(D_i))(\overline{\QQ})$, then we define the \emph{height} of $x$ with respect to $\overline{D}$ as
\[
 h_{\overline{D}}(x)=\frac{1}{[K(x):\QQ]}\left(\sum_{i=1}^la_i\log\sharp\left(\mathcal{O}_{C_x}(D_i)/\mathcal{O}_{C_x}\right)+\frac{1}{2}\sum_{\sigma:K(x)\to\CC}g_{\overline{D}}(x^{\sigma})\right).
\]
In general, we can define $h_{\overline{D}}(x)$ for any rational point $x\in X(\overline{\QQ})$ and for any arithmetic $\RR$-divisor $\overline{D}$ by expressing $\overline{D}$ as a difference of two arithmetic $\RR$-divisors each of which does not contain $x$ in its support (see \cite[\S 5.3]{Moriwaki12a} for details).
Let $\Rat(X)$ be the rational function field of $X$.
Associated to $\overline{D}:=(D,g_{\overline{D}})\in\aDiv_{\RR}(X;C^0)$, we have a $\ZZ$-module defined by
\[
 \Hz(X,D):=\{\phi\in\Rat(X)^{\times}\,|\,D+(\phi)\geq 0\}\cup\{0\},
\]
and a norm $\|\cdot\|_{\sup}^{\overline{D}}$ on $\Hz(X,D)_{\CC}:=\Hz(X,D)\otimes_{\ZZ}\CC$ defined by
\[
 \|\phi\|_{\sup}^{\overline{D}}:=\begin{cases}\sup_{x\in X(\CC)}\{|\phi|\exp(-g_{\overline{D}}/2)\} & \text{if $\phi\neq 0$,} \\ 0 & \text{if $\phi=0$}\end{cases}
\]
for $\phi\in\Hz(X,D)_{\CC}=\{\psi\in\Rat(X(\CC))^{\times}\,|\,D_{\CC}+(\psi)_{\CC}\geq 0\}\cup\{0\}$.
In other words, $\Hz(X,D)$ is defined as the $\ZZ$-module of global sections of $\mathcal{O}_X(\lfloor D\rfloor))$, where $\mathcal{O}_X(\lfloor D\rfloor)$ denotes the reflexive sheaf of rank one on $X$ associated to the round down $\lfloor D\rfloor$.
Note that the function
\[
 |\phi|_{\overline{D}}:=|\phi|\exp(-g_{\overline{D}}/2)
\]
is continuous on $X(\CC)$.
In fact, if we write $D=\sum_{i=1}^la_iD_i$ with $a_i\in\RR$ and effective Cartier divisors $D_i$ on $X$ and denote a local defining equation for $D_i$ by $f_i$, then we can see that near each point on $X(\CC)$ the rational function $\phi\cdot f_1^{\lfloor a_1\rfloor}\cdots f_l^{\lfloor a_l\rfloor}$ extends to a regular function.
Let $\pi:X'\to X$ be a surjective birational morphism of normal projective arithmetic varieties.
Then the natural homomorphism
\[
 \pi^*:(\Hz(X,D),\|\cdot\|_{\sup}^{\overline{D}})\xrightarrow{\sim}(\Hz(X',\pi^*D),\|\cdot\|_{\sup}^{\pi^*\overline{D}}),\quad \phi\mapsto\pi^*\phi,
\]
is an isometry.
We define $\ZZ$-submodules of $\Hz(X,D)$ by
\[
 \mathrm{F}^t(X,\overline{D}):=\left\langle\phi\in\Hz(X,D)\,\left|\,\|\phi\|_{\sup}^{\overline{D}}\leq\exp(-t)\right.\right\rangle_{\ZZ}
\]
and
\[
 \mathrm{F}^{t+}(X,\overline{D}):=\left\langle\phi\in\Hz(X,D)\,\left|\,\|\phi\|_{\sup}^{\overline{D}}<\exp(-t)\right.\right\rangle_{\ZZ}
\]
for $t\in\RR$.
For $\overline{D}\in\aDiv_{\RR}(X;C^0)$, we define the arithmetic volume of $\overline{D}$ as
\[
 \avol(\overline{D}):=\limsup_{m\to\infty}\frac{\log\sharp\{\phi\in\Hz(X,D)\,|\,\|\phi\|_{\sup}^{\overline{D}}\leq 1\}}{m^{d+1}/(d+1)!}.
\]
In \cite{Moriwaki12a}, Moriwaki proved that the volume function $\avol:\aDiv_{\RR}(X;C^0)\to\RR$ is continuous in the sense that
\[
 \lim_{\varepsilon_1,\dots,\varepsilon_r,\|f\|_{\sup}\to 0}\avol\left(\overline{D}+\sum_{i=1}^r\varepsilon_i\overline{E}_i+(0,f)\right)=\avol(\overline{D})
\]
for any arithmetic $\RR$-divisors $\overline{E}_1,\dots,\overline{E}_r$ and for any $f\in C^0(X)$.

\begin{lemma}\label{lem:defineachivol}
For any $f\in C^0(X)$, we have
\[
 |\avol(\overline{D}+(0,2f))-\avol(\overline{D})|\leq (d+1)\|f\|_{\sup}\vol(D_{\QQ}).
\]
\end{lemma}

\begin{remark}
The arithmetic divisor $(0,2f)$ corresponds to the Hermitian line bundle $(\mathcal{O}_X,\exp(-f)|\cdot|)$.
\end{remark}

\begin{proof}
This follows, for example, from \cite[Lemma~2.9]{Yuan09}.
\end{proof}

We recall some positivity notions for the arithmetic $\RR$-divisors.

\begin{itemize}
\item (\emph{ample}): Suppose that $X_{\QQ}$ is smooth.
$\overline{D}\in\aDiv_{\RR}(X;C^0)$ is said to be \emph{ample} if there exist arithmetic divisors of $C^{\infty}$-type, $\overline{A}_1,\dots,\overline{A}_l\in\aDiv(X;C^{\infty})$, such that (i) $A_i$ are ample, (ii) the curvature forms $\omega(\overline{A}_i)$ are positive point-wise on $X(\CC)$, and (iii) $\mathrm{F}^{0+}(X,m\overline{A}_i)=\Hz(X,mA_i)$ for all $m\gg 1$, and positive real numbers $a_1,\dots,a_l\in\RR_{>0}$ such that $\overline{D}=a_1\overline{A}_1+\dots+a_l\overline{A}_l$.
We say that $\overline{D}\in\aDiv_{\RR}(X;C^0)$ is \emph{adequate} if there exist an ample arithmetic $\RR$-divisor $\overline{A}$ and a non-negative continuous function $f\in C^0(X)$ such that $\overline{D}=\overline{A}+(0,f)$.
\item (\emph{nef}): Let $\overline{D}:=(D,g_{\overline{D}})\in\aDiv_{\KK}(X;\mathscr{T})$.
The Green function $g_{\overline{D}}$ is said to be \emph{plurisubharmonic} if $\pi^*g_{\overline{D}}$ is plurisubharmonic on $Y$ for one (and hence, for any) resolution of singularities $\pi:Y\to X(\CC)$.
We say that $\overline{D}$ is \emph{nef} if $D$ is relatively nef, $g_{\overline{D}}$ is plurisubharmonic, and $h_{\overline{D}}(x)\geq 0$ for every $x\in X(\overline{\QQ})$.
We denote the cone of all nef arithmetic $\KK$-divisors of $\mathscr{T}$-type by $\aNef_{\KK}(X;\mathscr{T})$, and denote the $\KK$-subspace of $\aDiv_{\KK}(X;\mathscr{T})$ generated by $\aNef_{\KK}(X;\mathscr{T})$ by $\aDiv_{\KK}^{\rm Nef}(X;\mathscr{T})$.
The elements of $\aDiv_{\KK}^{\rm Nef}(X;\mathscr{T})$ are usually referred to as \emph{integrable} arithmetic $\KK$-divisors.
\item (\emph{big}): $\overline{D}\in\aDiv_{\KK}(X;\mathscr{T})$ is said to be \emph{big} if $\avol(\overline{D})>0$.
We denote the cone of all big arithmetic $\KK$-divisors of $\mathscr{T}$-type by $\aBigCone_{\KK}(X;\mathscr{T})$.
Since an open convex cone in a finite dimensional $\RR$-vector space $\RR^r$ is generated by its rational points \cite[Theorem~6.3]{RockCA70}, the following two conditions are equivalent:
\begin{enumerate}
\item[(1)] $\overline{D}$ is big.
\item[(2)] There exist big arithmetic divisors $\overline{D}_1,\dots,\overline{D}_l$ and positive real numbers $a_1,\dots,a_l\in\RR_{>0}$ such that $\overline{D}=a_1\overline{D}_1+\dots+a_l\overline{D}_l$.
\end{enumerate}
\item (\emph{effective}): Let $\overline{D}:=(D,g_{\overline{D}})\in\aDiv_{\KK}(X;\mathscr{T})$.
We say that $\overline{D}$ is \emph{effective} if $\mult_{\Gamma}D\geq 0$ for all prime divisors $\Gamma$ on $X$ and $g_{\overline{D}}\geq 0$.
We write $\overline{D}\geq 0$ if $\overline{D}$ is effective.
\item (\emph{pseudo-effective}): We say that $\overline{D}\in\aDiv_{\RR}(X;C^0)$ is \emph{pseudo-effective} if, for any big arithmetic $\RR$-divisor $\overline{A}$, $\overline{D}+\overline{A}$ is big.
\end{itemize}

When $X$ is generically smooth and normal, Moriwaki \cite[\S 6.4]{Moriwaki12a} defined a map
\[
 \aDiv_{\RR}^{\rm Nef}(X;C^0)^{\times (d+1)}\to\RR,\quad (\overline{D}_0,\dots,\overline{D}_d)\mapsto\adeg(\overline{D}_0\cdots\overline{D}_d),
\]
which extends the usual arithmetic intersection product.
In the following, we show that one can define this map when $X$ is not necessarily generically smooth.

\begin{lemma}\label{lem:aint}
Let $\pi:X'\to X$ be a birational morphism of generically smooth normal projective arithmetic varieties.
Then
\[
 \adeg(\pi^*\overline{D}_0\cdots\pi^*\overline{D}_d)=\adeg(\overline{D}_0\cdots\overline{D}_d)
\]
for all $\overline{D}_0,\dots,\overline{D}_d\in\aDiv_{\RR}^{\rm Nef}(X;C^0)$.
\end{lemma}

\begin{proof}
If $\overline{D}_0,\dots,\overline{D}_d\in\aDiv_{\QQ}(X;C^{\infty})$, then the assertions are all clear (see the projection formula \cite[Proposition~2.4.1]{Kawaguchi_Moriwaki}).
In general, we may assume that $\overline{D}_0,\dots,\overline{D}_d\in\aNef_{\RR}(X;C^0)$.
Let $\varepsilon>0$ be a real number.
Let $\overline{H}_i$ be an ample arithmetic $\RR$-divisor such that $\overline{D}_i+\overline{H}_i\in\aNef_{\QQ}(X;C^0)$,
\[
 |\adeg((\overline{D}_0+\overline{H}_0)\cdots(\overline{D}_d+\overline{H}_d))-\adeg(\overline{D}_0\cdots\overline{D}_d)|<\varepsilon,
\]
and
\[
 |\adeg(\pi^*(\overline{D}_0+\overline{H}_0)\cdots\pi^*(\overline{D}_d+\overline{H}_d))-\adeg(\pi^*\overline{D}_0\cdots\pi^*\overline{D}_d)|<\varepsilon.
\]
By using \cite[Theorem~1]{Blocki_Kolo07} or \cite[Theorem~4.6]{Moriwaki12a}, one can find a non-negative function $f_i\in C^0(X)$ such that $\overline{D}_i+\overline{H}_i+(0,f_i)\in\aNef_{\QQ}(X;C^{\infty})$,
\begin{multline*}
 |\adeg((\overline{D}_0+\overline{H}_0+(0,f_0))\cdots(\overline{D}_d+\overline{H}_d+(0,f_d)))\\
 -\adeg((\overline{D}_0+\overline{H}_0)\cdots(\overline{D}_d+\overline{H}_d))|<\varepsilon,
\end{multline*}
and
\begin{multline*}
 |\adeg(\pi^*(\overline{D}_0+\overline{H}_0+(0,f_0))\cdots\pi^*(\overline{D}_d+\overline{H}_d+(0,f_d)))\\
 -\adeg(\pi^*(\overline{D}_0+\overline{H}_0)\cdots\pi^*(\overline{D}_d+\overline{H}_d))|<\varepsilon.
\end{multline*}
Since $\adeg(\pi^*(\overline{D}_0+\overline{H}_0+(0,f_0))\cdots\pi^*(\overline{D}_d+\overline{H}_d+(0,f_d)))=\adeg((\overline{D}_0+\overline{H}_0+(0,f_0))\cdots(\overline{D}_d+\overline{H}_d+(0,f_d)))$, we have
\[
 |\adeg(\pi^*\overline{D}_0\cdots\pi^*\overline{D}_d)-\adeg(\overline{D}_0\cdots\overline{D}_d)|<4\varepsilon
\]
for any $\varepsilon>0$.
\end{proof}

Suppose that $X$ is not generically smooth.
Let $\pi:X'\to X$ be a normalized generic resolution of singularities, and let $\overline{D}_0,\dots,\overline{D}_d\in\aDiv_{\RR}^{\rm Nef}(X;C^0)$.
Then $\pi^*\overline{D}_i\in\aDiv_{\RR}^{\rm Nef}(X';C^0)$ for all $i$.
We define the \emph{arithmetic intersection number} of $(\overline{D}_0,\dots,\overline{D}_d)$ as
\[
 \adeg(\overline{D}_0\cdots\overline{D}_d):=\adeg(\pi^*\overline{D}_0\cdots\pi^*\overline{D}_d),
\]
where the right-hand-side does not depend on the choice of $\pi$ by Lemma~\ref{lem:aint}.
By \cite[Proposition~6.4.2]{Moriwaki12a}, the map
\[
 \aDiv_{\RR}^{\rm Nef}(X;C^0)^{\times (d+1)}\to\RR,\quad (\overline{D}_0,\dots,\overline{D}_d)\mapsto\adeg(\overline{D}_0\cdots\overline{D}_d),
\]
is symmetric and multilinear and hence is also continuous: that is,
\begin{equation}
 \lim_{\varepsilon_{ij}\to 0}\adeg\left(\left(\overline{D}_0+\sum_{i=1}^{r_0}\varepsilon_{i0}\overline{E}_{i0}\right)\cdots\left(\overline{D}_d+\sum_{i=1}^{r_d}\varepsilon_{id}\overline{E}_{id}\right)\right)=\adeg(\overline{D}_0\cdots\overline{D}_d)
\end{equation}
for any $r_0,\dots,r_d\in\ZZ_{\geq 0}$ and for any integrable arithmetic $\RR$-divisors $\overline{E}_{10},\dots,\overline{E}_{r_dd}$.

\begin{lemma}\label{lem:aint2}
Let $X$ be a normal projective arithmetic variety.
(We do not assume that $X$ is generically smooth.)
\begin{enumerate}
\item[\textup{(1)}] If $\overline{D}_1,\dots,\overline{D}_d\in\aDiv_{\RR}^{\rm Nef}(X;C^0)$ and $\lambda\in\RR$, then
\[
 \adeg((0,2\lambda)\cdot\overline{D}_1\cdots\overline{D}_d)=\lambda\deg(D_{1,\QQ}\cdots D_{d,\QQ}).
\]
\item[\textup{(2)}] If $\overline{D}_1,\dots,\overline{D}_d\in\aNef_{\RR}(X;C^0)$ and $\overline{E}\in\aDiv_{\RR}^{\rm Nef}(X;C^0)$ is pseudo-effective, then
\[
 \adeg(\overline{E}\cdot\overline{D}_1\cdots\overline{D}_d)\geq 0.
\]
\item[\textup{(3)}] Let $\overline{D}_0,\dots,\overline{D}_d,\overline{E}_0,\dots,\overline{E}_d\in\aNef_{\RR}(X;C^0)$.
If $\overline{D}_i-\overline{E}_i$ is pseudo-effective for every $i$, then
\[
 \adeg(\overline{D}_0\cdots\overline{D}_d)\geq\adeg(\overline{E}_0\cdots\overline{E}_d).
\]
\end{enumerate}
\end{lemma}

\begin{proof}
(1) and (2) follow from the $C^{\infty}$ case as in Lemma~\ref{lem:aint}.

(3): By applying (2) successively, we have
\[
 \adeg(\overline{D}_0\cdots\overline{D}_d)\geq\adeg(\overline{E}_0\overline{D}_1\cdots\overline{D}_d)\geq\cdots\geq\adeg(\overline{E}_0\cdots\overline{E}_d).
\]
\end{proof}

\begin{lemma}\label{lem:aintkiso}
Let $X$ be a normal projective arithmetic variety.
(We do not assume that $X$ is generically smooth.)
The arithmetic intersection product uniquely extends to a multilinear map
\[
 \aDiv_{\RR}(X;C^0)\times\aDiv^{\rm Nef}_{\RR}(X;C^0)^{\times d}\to\RR,\quad (\overline{D}_0;\overline{D}_1,\dots,\overline{D}_d)\mapsto\adeg(\overline{D}_0\cdots\overline{D}_d),
\]
having the property that, if $\overline{D}_0$ is pseudo-effective and $\overline{D}_1,\dots,\overline{D}_d$ are nef, then
\[
 \adeg(\overline{D}_0\cdots\overline{D}_d)\geq 0.
\]
\end{lemma}

\begin{remark}
By the multilinearity, the above map is continuous in the sense that
\[
 \lim_{\varepsilon_1\to 0,\dots,\varepsilon_r\to 0}\adeg\left(\left(\overline{D}_0+\sum_{i=1}^r\varepsilon_i\overline{E}_i\right)\cdot\overline{D}_1\cdots\overline{D}_d\right)=\adeg(\overline{D}_0\cdots\overline{D}_d)
\]
for any arithmetic $\RR$-divisors $\overline{E}_1,\dots,\overline{E}_r$.
\end{remark}

\begin{proof}
We can assume that $X$ is generically smooth.
First, we assume that $\overline{D}_1,\dots,\overline{D}_d$ are nef.
We take a sequence of continuous functions $(f_n)_{n\geq 1}\subseteq C^0(X)$ such that $\|f_n\|_{\sup}\to 0$ as $n\to\infty$ and $\overline{D}_0+(0,f_n)\in\aDiv_{\RR}(X;C^{\infty})\subseteq\aDiv_{\RR}^{\rm Nef}(X;C^0)$ (in particular, $f_i-f_j$ is $C^{\infty}$ for every $i,j$).
Fix a nef and big $\RR$-divisor $A_{\QQ}$ such that $A_{\QQ}-D_{i,\QQ}$ are all pseudo-effective.
Since
\begin{align*}
 &|\adeg((\overline{D}_0+(0,f_i))\cdot\overline{D}_1\cdots\overline{D}_d)-\adeg((\overline{D}_0+(0,f_j))\cdot\overline{D}_1\cdots\overline{D}_d)|\\
 & \qquad\qquad =|\adeg((0,f_i-f_j)\cdot\overline{D}_1\cdots\overline{D}_d)|\leq\frac{1}{2}\deg(A_{\QQ}^{\cdot d})\cdot\|f_i-f_j\|_{\sup},
\end{align*}
the sequence $\left(\adeg((\overline{D}_0+(0,f_n))\cdot\overline{D}_1\cdots \overline{D}_{d}\right)_{n\geq 1}$ is a Cauchy sequence.
We set
\[
 \adeg(\overline{D}_0\cdot\overline{D}_{1}\cdots\overline{D}_d):=\lim_{n\to\infty}\adeg((\overline{D}_0+(0,f_n))\cdot\overline{D}_{1}\cdots\overline{D}_d),
\]
which does not depend on the choice of $(f_n)_{n\geq 1}$.
In general, we extend the map to $\aDiv_{\RR}(X;C^0)\times\aDiv^{\rm Nef}_{\RR}(X;C^0)^{\times d}\to\RR$ by using the multilinearity.

For the non-negativity, we choose the sequence $(f_n)_{n\geq 1}$ having the additional property that $f_n\geq 0$ for all $n$.
Then, by definition and Lemma~\ref{lem:aint2} (2), we have
\[
 \adeg(\overline{D}_0\cdot\overline{D}_{1}\cdots\overline{D}_d)=\lim_{n\to\infty}\adeg((\overline{D}_0+(0,f_n))\cdot\overline{D}_{1}\cdots\overline{D}_d)\geq 0.
\]
\end{proof}

The following is a version of the arithmetic Hodge index theorem (see \cite{Faltings84, Hriljac85, Moriwaki96, Yuan09, Yuan_Zhang13}).
The case where $\overline{H}=\overline{H}_1=\cdots=\overline{H}_{d-1}$ was treated by Yuan \cite{Yuan09}.

\begin{theorem}\label{thm:aHodge}
Let $X$ be a normal projective arithmetic variety of dimension $d+1$, and let $\overline{H}$, $\overline{H}_1,\dots,\overline{H}_{d-1}$ be nef arithmetic $\RR$-divisors on $X$.
Let $\overline{D}$ be an integrable arithmetic $\RR$-divisor on $X$.
\begin{enumerate}
\item[\textup{(1)}] Suppose that $H_{1,\QQ},\dots,H_{d-1,\QQ}$ are all big.
If $\deg(D_{\QQ}\cdot H_{1,\QQ}\cdots H_{d-1,\QQ})=0$, then $\adeg(\overline{D}^{\cdot 2}\cdot\overline{H}_1\cdots\overline{H}_{d-1})\leq 0$.
\item[\textup{(2)}] Suppose that $H_{\QQ},H_{1,\QQ},\dots,H_{d-1,\QQ}$ are all big.
If $\adeg(\overline{D}\cdot\overline{H}\cdot\overline{H}_1\cdots\overline{H}_{d-1})=0$, then $\adeg(\overline{D}^{\cdot 2}\cdot\overline{H}_1\cdots\overline{H}_{d-1})\leq 0$.
\end{enumerate}
\end{theorem}

\begin{remark}
There are many results in the literature on the equality conditions for Theorem~\ref{thm:aHodge} (1) (see \cite{Moriwaki96, Moriwaki10b}).
For example we can say that, if all $\overline{H}_i$ are ample and rational and if the equality holds in (1), then $D_{\QQ}$ is an $\RR$-linear combination of principal divisors on $X_{\QQ}$.
One can find a more precise equality condition for the above inequalities in Yuan-Zhang \cite[Theorem~1.3]{Yuan_Zhang13}.
In the following arguments, we do not use these equality conditions at least explicitly (but implicitly use in the proof of the general Dirichlet unit theorem \cite{Moriwaki10b}).
\end{remark}

\begin{proof}
This follows from Yuan-Zhang's version of the arithmetic Hodge index theorem \cite{Yuan_Zhang13}.
We may assume that $X$ is generically smooth.
Let $O_K:=\Hz(X,\mathcal{O}_X)$, where $K$ is an algebraic number field.

(1): First, we assume that $\overline{H}_1,\dots\overline{H}_{d-1}\in\aNef_{\QQ}(X;C^0)$.
We can find $\overline{D}_1,\dots,\overline{D}_l\in\aDiv(X;C^0)$ and $a_1,\dots,a_l\in\RR$ such that $a_1,\dots,a_l$ are linearly independent over $\QQ$ and
\[
 \overline{D}=a_1\overline{D}_1+\dots+a_l\overline{D}_l.
\]
Since $\sum_ia_i\deg(D_{i,\QQ}\cdot H_{1,\QQ}\cdots H_{d-1,\QQ})=0$ and $\deg(D_{i,\QQ}\cdot H_{1,\QQ}\cdots H_{d-1,\QQ})\in\QQ$, we have $\deg(D_{i,\QQ}\cdot H_{1,\QQ}\cdots H_{d-1,\QQ})=0$ for all $i$.
By Yuan-Zhang \cite{Yuan_Zhang13}, for any $\overline{E}\in\aDiv_{\QQ}^{\rm Nef}(X;C^0)$, if $\deg(E_{\QQ}\cdot H_{1,\QQ}\cdots H_{d-1,\QQ})=[K:\QQ]\deg(E_K\cdot H_{1,K}\cdots H_{d-1,K})=0$, then we have $\adeg(\overline{E}^{\cdot 2}\cdot\overline{H}_1\cdots\overline{H}_{d-1})\leq 0$.
Thus, we have
\[
 \adeg((b_1\overline{D}_1+\dots+b_l\overline{D}_l)^{\cdot 2}\cdot\overline{H}_1\cdots\overline{H}_{d-1})\leq 0
\]
for all $b_1,\dots,b_l\in\QQ$.
Therefore, we have $\adeg(\overline{D}^{\cdot 2}\cdot\overline{H}_1\cdots\overline{H}_{d-1})\leq 0$ by continuity.

Next, we fix an ample arithmetic divisor $\overline{A}$.
For each $i=1,\dots,d-1$, there exists a sequence of nef arithmetic $\RR$-divisors $(\overline{A}_i^{(j)})_{j=1}^{\infty}$ contained in a finite dimensional $\RR$-subspace $V$ of $\aDiv_{\RR}(X;C^0)$ such that $\overline{A}_i^{(j)}\to 0$ in $V$ as $j\to\infty$ and $\overline{H}_i^{(j)}:=\overline{H}_i+\overline{A}_i^{(j)}$ is rational for $j=1,2,\dots$.
Set
\[
 \varepsilon_j:=-\frac{\deg(D_{\QQ}\cdot H_{1,\QQ}^{(j)}\cdots H_{d-1,\QQ}^{(j)})}{\deg(A_{\QQ}\cdot H_{1,\QQ}^{(j)}\cdots H_{d-1,\QQ}^{(j)})}\in\RR
\]
for $j=1,2,\dots$.
Since $\deg((D_{\QQ}+\varepsilon_jA_{\QQ})\cdot H_{1,\QQ}^{(j)}\cdots H_{d-1,\QQ}^{(j)})=0$ and $\overline{H}_i^{(j)}\in\aNef_{\QQ}(X;C^0)$, we have
\[
 \adeg((\overline{D}+\varepsilon_j\overline{A})^{\cdot 2}\cdot\overline{H}_1^{(j)}\cdots\overline{H}_{d-1}^{(j)})\leq 0.
\]
As $j\to\infty$, we have $\overline{H}_i^{(j)}\to\overline{H}_i$ and
\[
 \varepsilon_j\to -\frac{\deg(D_{\QQ}\cdot H_{1,\QQ}\cdots H_{d-1,\QQ})}{\deg(A_{\QQ}\cdot H_{1,\QQ}\cdots H_{d-1,\QQ})}=0.
\]
Note that there exists a positive $N>0$ such that $\deg(A_{\QQ}\cdot H_{1,\QQ}\cdots H_{d-1,\QQ})\geq N\deg(A_{\QQ}^{\cdot d})>0$ since $H_{i,\QQ}$'s are all big.
Hence we have
\[
 \adeg(\overline{D}^{\cdot 2}\cdot\overline{H}_1\cdots\overline{H}_{d-1})\leq 0
\]
by continuity.

(2): Set $t:=\deg(D_{\QQ}\cdot H_{1,\QQ}\cdots H_{d-1,\QQ})/\deg(H_{\QQ}\cdot H_{1,\QQ}\cdots H_{d-1,\QQ})\in\RR$.
Since $\deg((D_{\QQ}-tH_{\QQ})\cdot H_{1,\QQ}\cdots H_{d-1,\QQ})=0$, we have
\begin{align*}
 &\adeg((\overline{D}-t\overline{H})^{\cdot 2}\cdot\overline{H}_1\cdots\overline{H}_{d-1})\\
 &\qquad\qquad=\adeg(\overline{D}^{\cdot 2}\cdot\overline{H}_1\cdots\overline{H}_{d-1})+t^2\adeg(\overline{H}^{\cdot 2}\cdot\overline{H}_1\cdots\overline{H}_{d-1})\leq 0.
\end{align*}
This means that $\adeg(\overline{D}^{\cdot 2}\cdot\overline{H}_1\cdots\overline{H}_{d-1})\leq 0$.
\end{proof}

The following series of inequalities is a formal consequence of Theorem~\ref{thm:aHodge} (see \cite[\S 1.6]{LazarsfeldI} for the original Khovanskii-Teissier inequalities in the context of algebraic geometry).

\begin{theorem}\label{thm:aKT}
Let $\overline{D},\overline{E},\overline{H}_0,\dots,\overline{H}_d\in\aNef_{\RR}(X;C^0)$.
\begin{enumerate}
\item[\textup{(1)}] $\adeg(\overline{D}\cdot\overline{E}\cdot\overline{H}_2\cdots\overline{H}_d)^2\geq\adeg(\overline{D}^{\cdot 2}\cdot\overline{H}_2\cdots\overline{H}_d)\cdot\adeg(\overline{E}^{\cdot 2}\cdot\overline{H}_2\cdots\overline{H}_d)$.
\item[\textup{(2)}] For any $k$ with $1\leq k\leq d+1$ and for any $i$ with $0\leq i\leq k$, we have
\[
 \adeg(\overline{D}^{\cdot i}\cdot\overline{E}^{\cdot (k-i)}\cdot\overline{H}_k\cdots\overline{H}_d)^k\geq\adeg(\overline{D}^{\cdot k}\cdot\overline{H}_k\cdots\overline{H}_d)^i\cdot\adeg(\overline{E}^{\cdot k}\cdot\overline{H}_k\cdots\overline{H}_d)^{k-i}.
\]
\item[\textup{(3)}] For any $k$ with $1\leq k\leq d+1$, we have
\[
 \adeg(\overline{H}_0\cdots\overline{H}_d)^k\geq\prod_{i=0}^{k-1}\adeg(\overline{H}_i^{\cdot k}\cdot\overline{H}_k\cdots\overline{H}_d).
\]
\item[\textup{(4)}] For any $k$ with $1\leq k\leq d+1$, we have
\[
 \adeg((\overline{D}+\overline{E})^{\cdot k}\cdot\overline{H}_k\cdots\overline{H}_d)^{1/k}\geq\adeg(\overline{D}^{\cdot k}\cdot\overline{H}_k\cdots\overline{H}_d)^{1/k}+\adeg(\overline{E}^{\cdot k}\cdot\overline{H}_k\cdots\overline{H}_d)^{1/k}.
\]
\end{enumerate}
\end{theorem}

\begin{remark}
By Theorem~\ref{thm:aKT} (1), we can see that the function $i\mapsto\log\adeg(\overline{D}^{\cdot i}\cdot\overline{E}^{\cdot (d-i+1)})$ is concave: that is, for any $i$ with $1\leq i\leq d$, we have
\[
 \adeg(\overline{D}^{\cdot i}\cdot\overline{E}^{\cdot (d-i+1)})^2\geq\adeg(\overline{D}^{\cdot (i-1)}\cdot\overline{E}^{\cdot (d-i+2)})\cdot\adeg(\overline{D}^{\cdot (i+1)}\cdot\overline{E}^{\cdot (d-i)}).
\]
\end{remark}

\begin{proof}
By adding a nef and big arithmetic $\RR$-divisor, we can assume that $\overline{D},\overline{E},\overline{H}_0,\dots,\overline{H}_d$ are all nef and big, and every arithmetic intersection number appearing below is positive.

(1): Set $\overline{F}:=\adeg(\overline{E}^{\cdot 2}\cdot\overline{H}_2\cdots\overline{H}_d)\overline{D}-\adeg(\overline{D}\cdot\overline{E}\cdot\overline{H}_2\cdots\overline{H}_d)\overline{E}\in\aDiv_{\RR}^{\rm Nef}(X;C^0)$.
Since $\adeg(\overline{F}\cdot\overline{E}\cdot\overline{H}_1\cdots\overline{H}_d)=0$, we have $\adeg(\overline{F}^{\cdot 2}\cdot\overline{H}_2\cdots\overline{H}_d)\leq 0$ by Theorem~\ref{thm:aHodge} (2).
This means that
\[
 \adeg(\overline{D}^{\cdot 2}\cdot\overline{H}_2\cdots\overline{H}_d)\cdot\adeg(\overline{E}^{\cdot 2}\cdot\overline{H}_2\cdots\overline{H}_d)\leq \adeg(\overline{D}\cdot\overline{E}\cdot\overline{H}_2\cdots\overline{H}_d)^2.
\]

(2): We prove the assertion by induction on $k$.
If $k=2$, then the assertion is nothing but (1).
In general, we may assume that $1\leq i\leq k-1$.
We have
\begin{align}
 &\adeg(\overline{D}^{\cdot i}\cdot\overline{E}^{\cdot (k-i)}\cdot\overline{H}_k\cdots\overline{H}_d) \label{eqn:aKT1}\\
 &\quad\geq\adeg(\overline{D}^{\cdot (k-1)}\cdot\overline{E}\cdot\overline{H}_k\cdots\overline{H}_d)^{i/(k-1)}\cdot\adeg(\overline{E}^{\cdot k}\cdot\overline{H}_k\cdots\overline{H}_d)^{(k-i-1)/(k-1)} \nonumber
\end{align}
(by the induction hypothesis) and
\begin{align}
 &\adeg(\overline{D}^{\cdot {(k-1)}}\cdot\overline{E}\cdot\overline{H}_k\cdots\overline{H}_d)^{2i/k} \label{eqn:aKT2}\\
 &\qquad\qquad\geq\adeg(\overline{D}^{\cdot k}\cdot\overline{H}_k\cdots\overline{H}_d)^{i/k}\cdot\adeg(\overline{D}^{\cdot (k-2)}\cdot\overline{E}^{\cdot 2}\cdot\overline{H}_k\cdots\overline{H}_d)^{i/k} \nonumber\\
 &\qquad\qquad\geq\adeg(\overline{D}^{\cdot k}\cdot\overline{H}_k\cdots\overline{H}_d)^{i/k}\cdot\adeg(\overline{E}^{\cdot k}\cdot\overline{H}_k\cdots\overline{H}_d)^{i/k(k-1)} \nonumber \\
 &\qquad\qquad\qquad\times\adeg(\overline{D}^{\cdot {(k-1)}}\cdot\overline{E}\cdot\overline{H}_k\cdots\overline{H}_d)^{i(k-2)/k(k-1)} \nonumber
\end{align}
(by using (1) for the first inequality and the induction hypothesis for the second).
By multiplying (\ref{eqn:aKT1}) by (\ref{eqn:aKT2}), we have
\[
 \adeg(\overline{D}^{\cdot i}\cdot\overline{E}^{\cdot (k-i)}\cdot\overline{H}_k\cdots\overline{H}_d)\geq\adeg(\overline{D}^{\cdot k}\cdot\overline{H}_k\cdots\overline{H}_d)^{i/k}\cdot\adeg(\overline{E}^{\cdot k}\cdot\overline{H}_k\cdots\overline{H}_d)^{(k-i)/k}.
\]
Note that the arithmetic intersection numbers we have considered are all assumed to be positive.

(3): We prove the assertion by induction on $k$.
If $k=2$, then the assertion is nothing but (1).
In general, we have
\begin{align*}
 \adeg(\overline{H}_0\cdots\overline{H}_d)&\geq\prod_{i=0}^{k-2}\adeg(\overline{H}_i^{\cdot (k-1)}\cdot\overline{H}_{k-1}\cdots\overline{H}_d)^{1/(k-1)}\\
 &\geq\prod_{i=0}^{k-2}\left(\adeg(\overline{H}_i^{\cdot k}\cdot\overline{H}_k\cdots\overline{H}_d)^{1/k}\cdot\adeg(\overline{H}_{k-1}^{\cdot k}\cdot\overline{H}_k\cdots\overline{H}_d)^{1/k(k-1)}\right)\\
 &=\prod_{i=0}^{k-1}\adeg(\overline{H}_i^{\cdot k}\cdot\overline{H}_k\cdots\overline{H}_d)^{1/k}
\end{align*}
by using (2).

(4): By (2), we have
\begin{align*}
 &\adeg((\overline{D}+\overline{E})^{\cdot k}\cdot\overline{H}_k\cdots\overline{H}_d)=\sum_{i=0}^k\binom{k}{i}\adeg(\overline{D}^{\cdot i}\cdot\overline{E}^{\cdot (k-i)}\cdot\overline{H}_k\cdots\overline{H}_d) \\
 &\qquad\geq\sum_{i=0}^k\binom{k}{i}\adeg(\overline{D}^{\cdot k}\cdot\overline{H}_k\cdots\overline{H}_d)^{i/k}\cdot\adeg(\overline{E}^{\cdot k}\cdot\overline{H}_k\cdots\overline{H}_d)^{(k-i)/k}\\
 &\qquad =\left(\adeg(\overline{D}^{\cdot k}\cdot\overline{H}_k\cdots\overline{H}_d)^{1/k}+\adeg(\overline{E}^{\cdot k}\cdot\overline{H}_k\cdots\overline{H}_d)^{1/k}\right)^k.
\end{align*}
\end{proof}

\section{Arithmetic positive intersection numbers}\label{sec:aposint}

Let $X$ be a normal projective arithmetic variety of dimension $d+1$, and let $\overline{D}$ be a big arithmetic $\RR$-divisor on $X$.
An \emph{approximation} of $\overline{D}$ is a pair $\overline{\mathcal{R}}:=(\varphi:X'\to X;\overline{M})$ consisting of a blowing up $\varphi:X'\to X$ and a nef arithmetic $\RR$-divisor $\overline{M}$ of $C^0$-type on $X'$ such that $X'$ is generically smooth and normal and $\overline{F}:=\varphi^*\overline{D}-\overline{M}$ is a pseudo-effective arithmetic $\RR$-divisor of $C^0$-type.
An approximation $(\varphi:X'\to X;\overline{M})$ of $\overline{D}$ is said to be \emph{admissible} if $\varphi^*\overline{D}-\overline{M}$ is an effective arithmetic $\QQ$-divisor of $C^0$-type.
Note that our terminology is slightly different from Chen's \cite[Definition~2]{Chen11}, which imposes the condition that $\overline{M}$ is semiample.
We denote the set of all approximations of $\overline{D}$ by $\widehat{\Theta}(\overline{D})$, and set
\begin{align*}
 &\widehat{\Theta}_{\rm ad}(\overline{D}):=\{\overline{\mathcal{R}}:=(\varphi:X'\to X;\overline{M})\in\widehat{\Theta}(\overline{D})\,|\,\text{$\overline{\mathcal{R}}$ is admissible}\},\\
 &\widehat{\Theta}_{C^{\infty}}(\overline{D}):=\{(\varphi:X'\to X;\overline{M})\in\widehat{\Theta}_{\rm ad}(\overline{D})\,|\,\text{$\overline{M}$ is $C^{\infty}$}\},\\
 &\widehat{\Theta}_{\rm amp}(\overline{D}):=\{(\varphi:X'\to X;\overline{M})\in\widehat{\Theta}_{C^{\infty}}(\overline{D})\,|\,\text{$\overline{M}$ is ample}\}.
\end{align*}
Let $n$ be an integer with $0\leq n\leq d$.
Let $\overline{D}_0,\dots,\overline{D}_n$ be big arithmetic $\RR$-divisors, $\overline{D}_{n+1},\dots,\overline{D}_d$ nef arithmetic $\RR$-divisors, and $\overline{\mathcal{R}}_i:=(\varphi_i:X_i'\to X;\overline{M}_i)\in\widehat{\Theta}(\overline{D}_i)$ for $i=0,\dots,n$.
We can choose a blow-up $\pi:X'\to X$ in such a way that $X'$ is generically smooth and normal and $\pi$ factors as $X'\xrightarrow{\psi_i}X_i'\xrightarrow{\varphi_i}X$ for each $i$.
Then we set
\begin{equation}
 \overline{\mathcal{R}}_0\cdots\overline{\mathcal{R}}_n\cdot\overline{D}_{n+1}\cdots\overline{D}_d:=\adeg(\psi_0^*\overline{M}_0\cdots\psi_n^*\overline{M}_n\cdot\pi^*\overline{D}_{n+1}\cdots\pi^*\overline{D}_d),
\end{equation}
which does not depend on the choice of $\pi:X'\to X$ by Lemma~\ref{lem:aint}.

\begin{proposition}\label{prop:defaposint1}
Suppose that $X$ is generically smooth and let $\overline{D}\in\aBigCone_{\RR}(X;C^0)$.
Let $\overline{D}=\overline{M}+\overline{F}$ be any decomposition such that $\overline{M}$ is a nef arithmetic $\RR$-divisor and that $\overline{F}$ is a pseudo-effective arithmetic $\RR$-divisor.
Let $\gamma$ be a real number with $0<\gamma<1$.
Then there exists a decomposition
\[
 \overline{D}=\overline{H}+\overline{E}
\]
such that $\overline{H}$ is an ample arithmetic $\RR$-divisor such that $\overline{H}-\gamma\overline{M}$ is a pseudo-effective arithmetic $\RR$-divisor and that $\overline{E}$ is an effective arithmetic $\QQ$-divisor.
In particular, the sets $\widehat{\Theta}_{\rm amp}(\overline{D})\subseteq\widehat{\Theta}_{C^{\infty}}(\overline{D})\subseteq\widehat{\Theta}_{\rm ad}(\overline{D})$ are all nonempty.
\end{proposition}

\begin{proof}
Since $\gamma\overline{D}=\gamma\overline{M}+\gamma\overline{F}$ and $(1-\gamma)\overline{D}$ is big, we can find a decomposition
\[
 \overline{D}=(2\overline{H}+\gamma\overline{M})+a_1\overline{E}_1+\dots+a_r\overline{E}_r+(0,2\delta)
\]
such that $\overline{H}$ is an ample arithmetic $\RR$-divisor, $a_1,\dots,a_r,\delta$ are positive real numbers, and $\overline{E}_1,\dots\overline{E}_r$ are big and effective arithmetic divisors.
Since $H_{\QQ}+\gamma M_{\QQ}$ is ample, we can approximate the metric of $\overline{H}+\gamma\overline{M}$ by smooth semipositive metrics (\cite[Theorem~1]{Blocki_Kolo07} or \cite[Theorem~4.6]{Moriwaki12a}).
Thus we can choose a non-negative continuous function $f\in C^0(X)$ such that $\|f\|_{\sup}<\delta$ and $\overline{H}+\gamma\overline{M}+(0,f)$ is a nef arithmetic $\RR$-divisor of $C^{\infty}$-type.
Moreover, by the Stone-Weierstrass theorem, we can find non-negative continuous functions $g_1,\dots,g_r\in C^0(X)$ such that $\|g_i\|_{\sup}<\delta/(a_1+\dots+a_r)$ and $\overline{E}_i+(0,g_i)$ is $C^{\infty}$ for all $i$.
Set $\overline{E}_i':=\overline{E}_i+(0,g_i)$ and $g:=a_1g_1+\dots+a_rg_r$.
Then
\[
 \overline{D}=(2\overline{H}+\gamma\overline{M}+(0,f))+a_1\overline{E}_1'+\dots+a_r\overline{E}_r'+(0,2\delta-f-g).
\]
Since $\overline{H}$ is ample and $\overline{E}_1',\dots,\overline{E}_r'$ are $C^{\infty}$, there exists an $\varepsilon>0$ such that
\[
 \overline{H}+\varepsilon_1\overline{E}_1'+\dots+\varepsilon_r\overline{E}_r'
\]
is ample for all $\varepsilon_1,\dots\varepsilon_r\in\RR$ with $|\varepsilon_1|+\dots+|\varepsilon_r|<\varepsilon$.
We can find $b_1,\dots,b_r\in\QQ_{>0}$ such that $|b_1-a_1|+\dots+|b_r-a_r|<\varepsilon$, and set $\overline{H}':=\overline{H}+(a_1-b_1)\overline{E}_1'+\dots+(a_r-b_r)\overline{E}_r'$.
Then $\overline{H}'$ is an ample arithmetic $\RR$-divisor, $b_1\overline{E}_1'+\dots+b_r\overline{E}_r'+(0,2\delta-f-g)$ is an effective arithmetic $\QQ$-divisor, and
\[
 \overline{D}=(\overline{H}'+\overline{H}+\gamma\overline{M}+(0,f))+b_1\overline{E}_1'+\dots+b_r\overline{E}_r'+(0,2\delta-f-g).
\]
Hence we conclude the proof.
\end{proof}

We define an order $\leq$ on the set $\widehat{\Theta}(\overline{D})$ in such a way that
\begin{align}
 &(\varphi_1:X_1'\to X;\overline{M}_1)\leq (\varphi_2:X_2'\to X;\overline{M}_2) \label{eqn:deforder}\\
 \overset{\rm def}{\Leftrightarrow}\quad &\text{there exists a blow-up $\varphi:X'\to X$ such that $\varphi$ factors as} \nonumber\\
 &\text{$X'\xrightarrow{\psi_i}X_i'\xrightarrow{\varphi_i} X$ for $i=1,2$ and $\psi_1^*\overline{M}_1\leq\psi_2^*\overline{M}_2$.}\nonumber
\end{align}
Then we have

\begin{proposition}\label{prop:defaposint2}
The set $\widehat{\Theta}_{\rm ad}(\overline{D})$ is filtered with respect to the order (\ref{eqn:deforder}).
\end{proposition}

\begin{proof}
Let $\overline{\mathcal{R}}_1:=(\varphi_1:Y_1\to X;\overline{M}_1)$ and $\overline{\mathcal{R}}_2:=(\varphi_2:Y_2\to X;\overline{M}_2)$ be two admissible approximations of $\overline{D}$ and set $\overline{F}_i:=\varphi_i^*\overline{D}-\overline{M}_i$ for $i=1,2$.
What we would like to show is that there exists an admissible approximation $\overline{\mathcal{R}}:=(\varphi:Y\to X;\overline{M})\in\widehat{\Theta}_{\rm ad}(\overline{D})$ such that $\overline{\mathcal{R}}_i\leq\overline{\mathcal{R}}$ for $i=1,2$.
By using the same arguments as above, we may assume that $Y_1=Y_2$ and $\varphi_1=\varphi_2$.
Let $m\geq 1$ be an integer such that $\overline{F}_1':=m\overline{F}_1$ (resp.\ $\overline{F}_2':=m\overline{F}_2$) has a non-zero section $s_1\in\Hz(Y_1,F_1')$ (resp.\ $s_2\in\Hz(Y_1,F_2')$) having supremum norm less than or equal to one.
Consider the morphism $\mathcal{O}_{Y_1}(-F_1')\oplus\mathcal{O}_{Y_1}(-F_2')\to\mathcal{O}_{Y_1}$ defined as $(t_1,t_2)\mapsto s_1\otimes t_1+s_2\otimes t_2$ for a local section $(t_1,t_2)$ of $\mathcal{O}_{Y_1}(-F_1')\oplus\mathcal{O}_{Y_1}(-F_2')$, and set
\[
 I:=\Image(\mathcal{O}_{Y_1}(-F_1')\oplus\mathcal{O}_{Y_1}(-F_2')\to\mathcal{O}_{Y_1}).
\]
Let $\psi_1:Y\to Y_1$ be a blowing up such that $Y$ is generically smooth and normal and that $\psi_1^{-1}I\cdot\mathcal{O}_Y$ is Cartier.
Let $\varphi:=\psi_1\circ\varphi_1$.
Let $F'$ be an effective Cartier divisor such that $\mathcal{O}_{Y}(-F')=\psi_1^{-1}I\cdot\mathcal{O}_Y$, and let $1_{F'}$ be the canonical section.
Then the assertion follows from Lemma~\ref{lem:filtered} below.
\end{proof}

\begin{lemma}\label{lem:bunbo}
Let $Y$ be a generically smooth normal projective arithmetic variety and $l\geq 1$ an integer.
For any $D\in\Div(Y)$ and for any $l\geq 1$, there exists a finite morphism $\psi:Z\to Y$ of arithmetic varieties and a Cartier divisor $D'\in\Div(Z)$ such that $Z$ is generically smooth and normal, and $\psi^*D\sim lD'$.
\end{lemma}

\begin{proof}
This is known as the Bloch-Gieseker covering trick, and \cite[Proof of Theorem~4.1.10]{LazarsfeldI} mutatis mutandis applies to our case (see also \cite[page 246, footnote]{LazarsfeldI}).
\end{proof}

\begin{lemma}\label{lem:filtered}
We keep the notations in Proposition~\ref{prop:defaposint2}.
\begin{enumerate}
\item[\textup{(1)}] We can endow $\mathcal{O}_{Y}(F')$ with a continuous Hermitian metric in such a way that
\[
 |1_{F'}|_{\overline{F}'}(x):=\max\left\{|s_1|_{\overline{F}_1'}(\psi_1(x)),|s_2|_{\overline{F}_2'}(\psi_1(x))\right\}\leq 1
\]
for $x\in Y(\CC)$, and $\psi_1^*\overline{F}_i'-\overline{F}'$ is effective for $i=1,2$.
\item[\textup{(2)}] Set $\overline{F}:=\overline{F}'/m$ and $\overline{M}:=\varphi^*\overline{D}-\overline{F}$.
Then $\overline{\mathcal{R}}:=(\varphi:Y\to X;\overline{M})\in\widehat{\Theta}_{\rm ad}(\overline{D})$ and $\overline{\mathcal{R}}_i\leq\overline{\mathcal{R}}$ for $i=1,2$.
\end{enumerate}
\end{lemma}

\begin{proof}
(1): We can choose an open covering $\{U_{\nu}\}$ of $Y(\CC)$ such that $\psi_1^*\mathcal{O}_Y(F_i')_{\CC}|_{U_{\nu}}$ is trivial with local frame $\eta_{i,\nu}$, and $F_{\CC}'\cap U_{\nu}$ is defined by a local equation $g_{\nu}$.
Since $s_i\in\Hz(Y_1,\mathcal{O}_{Y_1}(F_i')\otimes I)\subseteq\Hz(Y,\psi_1^*F_i'-F')$, there exists a $\sigma_i\in\Hz(Y,\psi_1^*F_i'-F')$ such that $\sigma_i\otimes 1_{F'}=\psi_1^*s_i$.
Thus, we can write
\[
 \psi_1^*s_i|_{U_{\nu}}=f_{i,\nu}\cdot g_{\nu}\cdot\eta_{i,\nu}
\]
on $U_{\nu}$, where $f_{1,\nu},f_{2,\nu}$ are holomorphic functions on $U_{\nu}$ satisfying $\{x\in U_{\nu}\,|\,f_{1,\nu}(x)=f_{2,\nu}(x)=0\}=\emptyset$.
Since
\begin{align*}
 &\max\left\{|s_1|_{\overline{F}_1'}(\psi_1(x)),|s_2|_{\overline{F}_2'}(\psi_1(x))\right\}\\
 &\qquad\qquad=\max\left\{|f_{1,\nu}(x)|\cdot |\eta_{1,\nu}|_{\psi_1^*\overline{F}_1'}(x),|f_{2,\nu}(x)|\cdot |\eta_{2,\nu}|_{\psi_1^*\overline{F}_2'}(x)\right\}\cdot |g_{\nu}(x)|
\end{align*}
for $x\in U_{\nu}$, we have the first half of the assertion.
The latter half follows from
\[
 \|\sigma_i\|_{\sup}^{\psi_1^*\overline{F}_i'-\overline{F}'}=\sup_{x\in (Y\setminus F')(\CC)}\frac{|s_i|_{\overline{F}_i'}(\psi_1(x))}{\max_j\left\{|s_j|_{\overline{F}_j'}(\psi_1(x))\right\}}\leq 1.
\]

(2): Since $\varphi^*g_{\overline{D}}-\psi_1^*g_{\overline{F}_i'}/m$ are plurisubharmonic, so is
\[
 g_{\overline{M}}:=\max\left\{\varphi^*g_{\overline{D}}-\frac{1}{m}\psi_1^*g_{\overline{F}_1'},\varphi^*g_{\overline{D}}-\frac{1}{m}\psi_1^*g_{\overline{F}_2'}\right\}.
\]
Let $\overline{H}$ be any ample arithmetic $\RR$-divisor on $Y$ such that $\overline{E}:=\varphi^*\overline{D}+\overline{H}$ is an arithmetic $\QQ$-divisor.
Set $\overline{N}_i:=\psi_1^*\overline{M}_i+\overline{H}=\overline{E}-\psi_1^*\overline{F}_i\in\aNef_{\QQ}(Y;C^0)$ for $i=1,2$ and $\overline{N}:=\overline{M}+\overline{H}=\overline{E}-\overline{F}\in\aDiv_{\QQ}(Y;C^0)$.
By Lemma~\ref{lem:bunbo}, we have

\begin{claim}
Let $l\geq 1$ be an integer such that $lmN_1$, $lmN_2$, and $lmN$ are all Cartier divisors on $Y$.
Then there exists a finite morphism $\psi:Z\to Y$ of arithmetic varieties and Cartier divisors $N_1'$, $N_2'$, and $N'$ on $Z$ such that $Z$ is normal and generically smooth and $lmN_1\sim lN_1'$, $lmN_2\sim lN_2'$, and $lmN\sim lN'$.
\end{claim}

We set $\overline{N}_1'$ (resp.\ $\overline{N}_2'$, $\overline{N}'$) as $N_1'$ (resp.\ $N_2'$, $N'$) endowed with the Green function induced from $m\psi^*\overline{N}_1$ (resp.\ $m\psi^*\overline{N}_2$, $m\psi^*\overline{N}$).
Then $\overline{N}_1',\overline{N}_2'\in\aNef(Z;C^0)$, $\overline{N}'\in\aDiv(Z;C^0)$, and $N_1'$ and $N_2'$ are ample.
Since the morphism $\mathcal{O}_Y(-\psi_1^*F_1')\oplus\mathcal{O}_Y(-\psi_1^*F_2')\to \mathcal{O}_Y(-F')$ is surjective, we have a surjective morphism $\mathcal{O}_Z(N_1')\oplus \mathcal{O}_Z(N_2')\to \mathcal{O}_Z(N')$ sending a local section  $(t_1,t_2)$ to $t_1\otimes\psi^*\sigma_1+t_2\otimes\psi^*\sigma_2$.

\begin{claim}\label{clm:tsuke}
For every sufficiently large $p\geq 1$ and for every $k=0,1,\dots,p$, $\mathcal{O}_Z(kN_1'+(p-k)N_2')$ is generated by its global sections.
In particular, $\Sym^p(N_1'\oplus N_2')$ is generated by its global sections for every $p\gg 1$.
\end{claim}

\begin{proof}
Since $N_1'$ and $N_2'$ are ample, there exists a $k_0\gg 1$ such that
\[
 \mathcal{O}_Z(pN_1')\quad\text{and}\quad \mathcal{O}_Z(pN_2')
\]
are globally generated for every $p\geq k_0$.
For $q=0,1,\dots,k_0-1$, there exists an $l_0\gg 1$ such that
\[
 \mathcal{O}_Z(pN_1'+qN_2')\quad\text{and}\quad\mathcal{O}_Z(qN_1'+pN_2')
\]
are globally generated for every $p\geq l_0$.
Suppose that $p+q\geq k_0+l_0$.
If $p\geq k_0$ and $q\geq k_0$, then $\mathcal{O}_Z(pN_1'+qN_2')$ is globally generated.
If $p<k_0$ (resp.\ $q<k_0$), then $q\geq l_0$ (resp.\ $p\geq l_0$) and $\mathcal{O}_Z(pN_1'+qN_2')$ is globally generated.
Hence we conclude.
\end{proof}

Since the diagram
\[
\xymatrix{
 \Hz(Z,pN')\otimes_{\ZZ}\mathcal{O}_Z \ar[r] & pN' \\
 \bigoplus_{k=0}^p\Hz(Z,kN_1'+(p-k)N_2')\otimes_{\ZZ}\mathcal{O}_Z \ar[r] \ar[u] & \Sym^p(N_1'\oplus N_2'), \ar[u]
}
\]
is commutative, we can see that $N'$ is nef.

\begin{claim}\label{clm:chottomendo}
For every sufficiently large $p\geq 1$ and for every $k=0,1,\dots,p$, we have
\[
\mathrm{F}^{0+}(Z,k\overline{N}_1'+(p-k)\overline{N}_2')_{\QQ}=\Hz(Z,kN_1'+(p-k)N_2')_{\QQ}.
\]
\end{claim}

\begin{proof}
Since $\overline{N}_1$ and $\overline{N}_2$ are both adequate on $Y$, there exists a $k_0\gg 1$ such that $\mathrm{F}^{0+}(Z,p\overline{N}_1')_{\QQ}=\Hz(Z,pN_1')_{\QQ}$ and $\mathrm{F}^{0+}(Z,p\overline{N}_2')_{\QQ}=\Hz(Z,pN_2')_{\QQ}$ for every $p\geq k_0$, and $\Hz(Z,pN_1')_{\QQ}\otimes\Hz(Z,qN_2')_{\QQ}\to\Hz(Z,pN_1'+qN_2')_{\QQ}$ is surjective for every $p,q$ with $p\geq k_0$ and $q\geq k_0$.
One can find an $l_0\gg 1$ such that $\mathrm{F}^{0+}(Z,p\overline{N}_1'+q\overline{N}_2')_{\QQ}=\Hz(Z,pN_1'+qN_2')_{\QQ}$ and $\mathrm{F}^{0+}(Z,q\overline{N}_1'+p\overline{N}_2')_{\QQ}=\Hz(Z,qN_1'+pN_2')_{\QQ}$ for every $p\geq l_0$ and for every $q=0,1,\dots,k_0-1$.
Then the claim holds for all $p\geq k_0+l_0$.
\end{proof}

We choose a $p\gg 1$ as in Claims~\ref{clm:tsuke} and \ref{clm:chottomendo}.
Since $\mathrm{F}^{0+}(Z,p\overline{N}')\otimes_{\ZZ}\mathcal{O}_{Z_{\QQ}}\to pN_{\QQ}'$ is surjective, $\overline{N}'$ is nef and thus $\overline{N}=\overline{M}+\overline{H}$ is also nef.
\end{proof}

For $\overline{D}_0,\dots,\overline{D}_n\in\aBigCone_{\RR}(X;C^0)$ and $\overline{D}_{n+1},\dots,\overline{D}_d\in\aNef_{\RR}(X;C^0)$, we define the \emph{arithmetic positive intersection number} of $(\overline{D}_0,\dots,\overline{D}_n;\overline{D}_{n+1},\dots,\overline{D}_d)$ as
\begin{equation}
 \langle\overline{D}_0\cdots\overline{D}_n\rangle\overline{D}_{n+1}\cdots\overline{D}_d:=\sup_{\overline{\mathcal{R}}_i\in\widehat{\Theta}_{\rm ad}(\overline{D}_i)}\overline{\mathcal{R}}_0\cdots\overline{\mathcal{R}}_n\cdot\overline{D}_{n+1}\cdots\overline{D}_d,
\end{equation}
where the supremum is taken over all admissible approximations $\overline{\mathcal{R}}_i\in\widehat{\Theta}_{\rm ad}(\overline{D}_i)$ for $i=0,1,\dots,n$.

\begin{remark}\label{rem:aposinthomo}
\begin{enumerate}
\item[(1)] By Proposition~\ref{prop:defaposint2}, the map
\begin{gather*}
 \aBigCone_{\RR}(X;C^0)^{\times (n+1)}\times\aNef_{\RR}(X;C^0)^{\times (d-n)}\to\RR,\\
 (\overline{D}_0,\dots,\overline{D}_n;\overline{D}_{n+1},\dots,\overline{D}_d)\mapsto\langle\overline{D}_0\cdots\overline{D}_n\rangle\overline{D}_{n+1}\cdots\overline{D}_d,
\end{gather*}
is symmetric and multilinear in the variables $\overline{D}_{n+1},\dots,\overline{D}_d$, and symmetric and positively homogeneous of degree one in $\overline{D}_0,\dots,\overline{D}_n$ and in $\overline{D}_{n+1},\dots,\overline{D}_d$.
In particular, by using the multilinearity, we can extend it to a map
\[
 \aBigCone_{\RR}(X;C^0)^{\times (n+1)}\times\aDiv_{\RR}^{\rm Nef}(X;C^0)^{\times (d-n)}\to\RR,
\]
which we also denote by $(\overline{D}_0,\dots,\overline{D}_n;\overline{D}_{n+1},\dots,\overline{D}_d)\mapsto\langle\overline{D}_0\cdots\overline{D}_n\rangle\overline{D}_{n+1}\cdots\overline{D}_d$.
\item[(2)] Let $\overline{D}_0,\dots,\overline{D}_n$ be big arithmetic $\RR$-divisors, $k_0,\dots,k_n$ positive integers with $k_0+\dots+k_n=N+1$, and $\overline{D}_{N+1},\dots,\overline{D}_d$ nef arithmetic $\RR$-divisors.
Then by Proposition~\ref{prop:defaposint2}, we have
\[
 \langle\overline{D}_0^{\cdot k_0}\cdots\overline{D}_n^{\cdot k_n}\rangle\overline{D}_{N+1}\cdots\overline{D}_d:=\sup_{\overline{\mathcal{R}}_i\in\widehat{\Theta}_{\rm ad}(\overline{D}_i)}\overline{\mathcal{R}}_0^{\cdot k_0}\cdots\overline{\mathcal{R}}_n^{\cdot k_n}\cdot\overline{D}_{N+1}\cdots\overline{D}_d.
\]
\item[(3)] If $\overline{D}_n$ is big and nef, then
\[
 \langle\overline{D}_0\cdots\overline{D}_n\rangle\overline{D}_{n+1}\cdots\overline{D}_d=\langle\overline{D}_0\cdots\overline{D}_{n-1}\rangle\overline{D}_n\cdot\overline{D}_{n+1}\cdots\overline{D}_d.
\]
\end{enumerate}
\end{remark}

\begin{proposition}\label{prop:replacewithample}
Let $\overline{D}_0,\dots,\overline{D}_n$ be big arithmetic $\RR$-divisors, $k_0,\dots,k_n$ positive integers with $k_0+\dots+k_n=N+1$, and $\overline{D}_{N+1},\dots,\overline{D}_d$ nef and big arithmetic $\RR$-divisors.
Then we have
\begin{align*}
 \langle\overline{D}_0^{\cdot k_0}\cdots\overline{D}_n^{\cdot k_n}\rangle\overline{D}_{N+1}\cdots\overline{D}_d&=\sup_{\overline{\mathcal{R}}_i\in\widehat{\Theta}(\overline{D}_i)}\overline{\mathcal{R}}_0^{\cdot k_0}\cdots\overline{\mathcal{R}}_n^{\cdot k_n}\cdot\overline{D}_{N+1}\cdots\overline{D}_d\\
 &=\sup_{\overline{\mathcal{R}}_i\in\widehat{\Theta}_{C^{\infty}}(\overline{D}_i)}\overline{\mathcal{R}}_0^{\cdot k_0}\cdots\overline{\mathcal{R}}_n^{\cdot k_n}\cdot\overline{D}_{N+1}\cdots\overline{D}_d\\
 &=\sup_{\overline{\mathcal{R}}_i\in\widehat{\Theta}_{\rm amp}(\overline{D}_i)}\overline{\mathcal{R}}_0^{\cdot k_0}\cdots\overline{\mathcal{R}}_n^{\cdot k_n}\cdot\overline{D}_{N+1}\cdots\overline{D}_d.
\end{align*}
\end{proposition}

\begin{proof}
The inequalities
\begin{align*}
 \sup_{\overline{\mathcal{R}}_i\in\widehat{\Theta}(\overline{D}_i)}\overline{\mathcal{R}}_0^{\cdot k_0}\cdots\overline{\mathcal{R}}_n^{\cdot k_n}\cdot\overline{D}_{N+1}\cdots\overline{D}_d&\geq\langle\overline{D}_0^{\cdot k_0}\cdots\overline{D}_n^{\cdot k_n}\rangle\overline{D}_{N+1}\cdots\overline{D}_d\\
 &\geq\sup_{\overline{\mathcal{R}}_i\in\widehat{\Theta}_{C^{\infty}}(\overline{D}_i)}\overline{\mathcal{R}}_0^{\cdot k_0}\cdots\overline{\mathcal{R}}_n^{\cdot k_n}\cdot\overline{D}_{N+1}\cdots\overline{D}_d\\
 &\geq\sup_{\overline{\mathcal{R}}_i\in\widehat{\Theta}_{\rm amp}(\overline{D}_i)}\overline{\mathcal{R}}_0^{\cdot k_0}\cdots\overline{\mathcal{R}}_n^{\cdot k_n}\cdot\overline{D}_{N+1}\cdots\overline{D}_d>0
\end{align*}
are trivial.
Let $\varepsilon>0$ be a sufficiently small positive real number and fix an approximation $\overline{\mathcal{R}}_i:=(\varphi:X'\to X;\overline{M}_i)\in\widehat{\Theta}(\overline{D}_i)$ for $i=0,1,\dots,n$ such that
\begin{equation}\label{eqn:adchoice}
 \overline{\mathcal{R}}_0^{\cdot k_0}\cdots\overline{\mathcal{R}}_n^{\cdot k_n}\cdot\overline{D}_{N+1}\cdots\overline{D}_d\geq\sup_{\overline{\mathcal{R}}_i\in\widehat{\Theta}(\overline{D}_i)}\overline{\mathcal{R}}_0^{\cdot k_0}\cdots\overline{\mathcal{R}}_n^{\cdot k_n}\cdot\overline{D}_{N+1}\cdots\overline{D}_d-\varepsilon>\varepsilon.
\end{equation}
Let $\gamma$ be a positive rational number such that $0<\gamma<1$ and
\begin{align}
 &\adeg((\gamma\overline{M}_0)^{\cdot k_0}\cdots(\gamma\overline{M}_n)^{\cdot k_n}\cdot\varphi^*\overline{D}_{N+1}\cdots\varphi^*\overline{D}_d) \label{eqn:amplegamma}\\
 &\qquad\qquad\qquad\qquad\qquad\qquad\qquad\geq\overline{\mathcal{R}}_0^{\cdot k_0}\cdots\overline{\mathcal{R}}_n^{\cdot k_n}\cdot\overline{D}_{N+1}\cdots\overline{D}_d-\varepsilon>0.\nonumber
\end{align}
By Proposition~\ref{prop:defaposint1}, we can find $\overline{\mathcal{R}}_i':=(\varphi:X'\to X;\overline{H}_i)\in\widehat{\Theta}_{\rm amp}(\overline{D}_i)$ such that $\overline{H}_i-\gamma\overline{M}_i$ is pseudo-effective.
Thus by using Lemma~\ref{lem:aint2} (3), we have
\begin{align}
 &{\overline{\mathcal{R}}_0'}^{\cdot k_0}\cdots{\overline{\mathcal{R}}_n'}^{\cdot k_n}\cdot\overline{D}_{N+1}\cdots\overline{D}_d \label{eqn:ampleamplechoice}\\
 &\qquad\qquad\qquad\qquad\geq\adeg((\gamma\overline{M}_0)^{\cdot k_0}\cdots(\gamma\overline{M}_n)^{\cdot k_n}\cdot\varphi^*\overline{D}_{N+1}\cdots\varphi^*\overline{D}_d). \nonumber
\end{align}
By (\ref{eqn:adchoice}), (\ref{eqn:amplegamma}), and (\ref{eqn:ampleamplechoice}), we have
\[
 \sup_{\overline{\mathcal{R}}_i'\in\widehat{\Theta}_{\rm amp}(\overline{D}_i)}{\overline{\mathcal{R}}_0'}^{\cdot k_0}\cdots{\overline{\mathcal{R}}_n'}^{\cdot k_n}\cdot\overline{D}_{N+1}\cdots\overline{D}_d\geq\sup_{\overline{\mathcal{R}}_i\in\widehat{\Theta}(\overline{D}_i)}\overline{\mathcal{R}}_0^{\cdot k_0}\cdots\overline{\mathcal{R}}_n^{\cdot k_n}\cdot\overline{D}_{N+1}\cdots\overline{D}_d-2\varepsilon
\]
for all $\varepsilon>0$.
This completes the proof of the proposition.
\end{proof}

\begin{proposition}\label{prop:aposcont}
\begin{enumerate}
\item[\textup{(1)}] Let $\overline{D}_0,\dots,\overline{D}_n,\overline{E}_0,\dots,\overline{E}_n$ be big arithmetic $\RR$-divisors, and let $\overline{D}_{n+1},\dots,\overline{D}_d,\overline{E}_{n+1},\dots,\overline{E}_d$ be nef and big arithmetic $\RR$-divisors.
If $\overline{D}_i-\overline{E}_i$ is pseudo-effective for every $i$, then
\[
 \langle\overline{D}_0\cdots\overline{D}_n\rangle\overline{D}_{n+1}\cdots\overline{D}_d\geq\langle\overline{E}_0\cdots\overline{E}_n\rangle\overline{E}_{n+1}\cdots\overline{E}_d.
\]
\item[\textup{(2)}] The map
\begin{gather*}
 \aBigCone_{\RR}(X;C^0)^{\times (n+1)}\times\aDiv_{\RR}^{\rm Nef}(X;C^0)^{\times (n-d)}\to\RR,\\
 (\overline{D}_0,\dots,\overline{D}_n;\overline{D}_{n+1},\dots,\overline{D}_d)\mapsto\langle\overline{D}_0\cdots\overline{D}_n\rangle\overline{D}_{n+1}\cdots\overline{D}_d,
\end{gather*}
is continuous in the sense that
\begin{align*}
 &\lim_{\varepsilon_{ij},\|f_j\|_{\sup}\to 0}\left\langle\left(\overline{D}_0+\sum_{i=1}^{r_0}\varepsilon_{i0}\overline{E}_{i0}+(0,f_0)\right)\cdots\left(\overline{D}_n+\sum_{i=1}^{r_n}\varepsilon_{in}\overline{E}_{in}+(0,f_n)\right)\right\rangle\\
 &\qquad\qquad\qquad\qquad\qquad\qquad\qquad\qquad\cdot\overline{D}_{n+1}\cdots\overline{D}_d=\langle\overline{D}_0\cdots\overline{D}_n\rangle\overline{D}_{n+1}\cdots\overline{D}_d
\end{align*}
for any $r_0,\dots,r_n\in\ZZ_{\geq 0}$, $\overline{E}_{10},\dots,\overline{E}_{r_nn}\in\aDiv_{\RR}(X;C^0)$, and $f_0,\dots,f_n\in C^0(X)$.
\item[\textup{(3)}] Suppose that $n=d-1$.
The map
\begin{gather*}
 \aBigCone_{\RR}(X;C^0)^{\times d}\times\aDiv^{\rm Nef}_{\RR}(X;C^0)\to\RR,\\
 (\overline{D}_0,\dots,\overline{D}_{d-1};\overline{D}_d)\mapsto\langle \overline{D}_0\cdots \overline{D}_{d-1}\rangle \overline{D}_d,
\end{gather*}
uniquely extends to a continuous map $\aBigCone_{\RR}(X;C^0)^{\times d}\times\aDiv_{\RR}(X;C^0)\to\RR$, which we also denote by $(\overline{D}_0,\dots,\overline{D}_{d-1};\overline{D}_d)\mapsto\langle \overline{D}_0\cdots \overline{D}_{d-1}\rangle \overline{D}_d$.
\item[\textup{(4)}] Let $\overline{D}_1,\dots,\overline{D}_d\in\aBigCone_{\RR}(X;C^0)$ and $\overline{E}\in\aDiv_{\RR}(X;C^0)$.
If $\overline{E}$ is pseudo-effective, then we have $\langle \overline{D}_1\cdots \overline{D}_d\rangle \overline{E}\geq 0$.
\end{enumerate}
\end{proposition}

\begin{proof}
(1): Since $\widehat{\Theta}(\overline{D}_i)\supseteq\widehat{\Theta}(\overline{E}_i)$ for $i=0,1,\dots,n$, the assertion follows from Lemma~\ref{lem:aint2} (3).

(2): We can assume that $\overline{D}_{n+1},\dots,\overline{D}_d$ are all nef.
Moreover, by using (1), we can assume that $f_0,\dots,f_n$ are all zero functions.
Suppose that $\varepsilon_{ij}$ are all sufficiently small.
Then by (1) and the homogeneity (Remark~\ref{rem:aposinthomo} (1)), we can choose a sufficiently small $\gamma$ with $0<\gamma<1$ such that
\begin{multline*}
 (1-\gamma)^n\langle\overline{D}_0\cdots\overline{D}_n\rangle\overline{D}_{n+1}\cdots\overline{D}_d\\
 \leq\left\langle\left(\overline{D}_0+\sum_{i=1}^{r_0}\varepsilon_{i0}\overline{E}_{i0}\right)\cdots\left(\overline{D}_n+\sum_{i=1}^{r_n}\varepsilon_{in}\overline{E}_{in}\right)\right\rangle\overline{D}_{n+1}\cdots\overline{D}_d\\
 \leq(1+\gamma)^n\langle\overline{D}_0\cdots\overline{D}_n\rangle\overline{D}_{n+1}\cdots\overline{D}_d
\end{multline*}
(see \cite[Proof of Proposition~2.9]{Bou_Fav_Mat06} and \cite[Proof of Proposition~3.6]{Chen11}).
Hence we conclude.

(3), (4): We can use the same argument as in Lemma~\ref{lem:aintkiso} (see \cite[\S 3.3, Remark~8]{Chen11}).
\end{proof}

\begin{proposition}\label{prop:avollogconcave}
\begin{enumerate}
\item[\textup{(1)}] For $\overline{D}\in\aBigCone_{\RR}(X;C^0)$, we have $\avol(\overline{D})=\langle\overline{D}^{\cdot (d+1)}\rangle$.
\item[\textup{(2)}] Let $\overline{D},\overline{E}\in\aBigCone_{\RR}(X;C^0)$.
We have
\[
 \avol(\overline{D}+\overline{E})\geq\sum_{i=0}^{d+1}\binom{d+1}{i}\langle\overline{D}^{\cdot i}\cdot\overline{E}^{d-i+1}\rangle.
\]
\item[\textup{(3)}] Let $\overline{D},\overline{E}\in\aBigCone_{\RR}(X;C^0)$.
Then the function $i\mapsto\log\langle\overline{D}^{\cdot i}\cdot\overline{E}^{d-i+1}\rangle$ is concave: that is, for any $i$ with $1\leq i\leq d$, we have
\[
 \langle\overline{D}^{\cdot i}\cdot\overline{E}^{d-i+1}\rangle^2\geq \langle\overline{D}^{\cdot i-1}\cdot\overline{E}^{\cdot d-i+2}\rangle\cdot\langle\overline{D}^{\cdot i+1}\cdot\overline{E}^{\cdot d-i}\rangle.
\]
In particular, we have
\[
 \langle\overline{D}^{\cdot i}\cdot\overline{E}^{d-i+1}\rangle\geq\avol(\overline{D})^\frac{i}{d+1}\cdot\avol(\overline{E})^{\frac{d-i+1}{d+1}}
\]
for $i$ with $1\leq i\leq d-1$, and
\[
 \langle\overline{D}^{\cdot d}\rangle\overline{E}\geq\langle\overline{D}^{\cdot d}\cdot\overline{E}\rangle\geq\avol(\overline{D})^\frac{d}{d+1}\cdot\avol(\overline{E})^{\frac{1}{d+1}}.
\]
\item[\textup{(4)}] Let $\overline{D},\overline{E},\overline{D}_k,\dots,\overline{D}_n\in\aBigCone_{\RR}(X;C^0)$ and $\overline{D}_{n+1},\dots,\overline{D}_d\in\aNef_{\RR}(X;C^0)$.
Then we have
\begin{align*}
 &\Bigl(\langle (\overline{D}+\overline{E})^{\cdot k}\cdot\overline{D}_k\cdots\overline{D}_n\rangle \overline{D}_{n+1}\cdots\overline{D}_d\Bigr)^{\frac{1}{k}} \\
 &\qquad\geq\left(\langle\overline{D}^{\cdot k}\cdot\overline{D}_k\cdots\overline{D}_n\rangle\overline{D}_{n+1}\cdots\overline{D}_d\right)^{\frac{1}{k}}+\left(\langle\overline{E}^{\cdot k}\cdot\overline{D}_k\cdots\overline{D}_n\rangle\overline{D}_{n+1}\cdots\overline{D}_d\right)^{\frac{1}{k}}.
\end{align*}
\end{enumerate}
\end{proposition}

\begin{proof}
(1): The inequality $\avol(\overline{D})\geq\langle\overline{D}^{\cdot (d+1)}\rangle$ is clear.
For any $\varepsilon>0$, one can find a big arithmetic $\QQ$-divisor $\overline{D}'$ such that $\overline{D}-\overline{D}'$ is effective and
\begin{equation}\label{eqn:afujita1}
 \avol(\overline{D}')+\varepsilon\geq\avol(\overline{D}).
\end{equation}
By the arithmetic Fujita approximation \cite{Chen11, Yuan09}, there exists an admissible approximation $(\varphi;\overline{M})\in\widehat{\Theta}_{\rm ad}(\overline{D}')$ such that
\begin{equation}\label{eqn:afujita2}
 \langle{\overline{D}'}^{\cdot (d+1)}\rangle+\varepsilon\geq\avol(\overline{M})+\varepsilon\geq\avol(\overline{D}').
\end{equation}
By (\ref{eqn:afujita1}) and (\ref{eqn:afujita2}), we have $\langle\overline{D}^{\cdot (d+1)}\rangle+2\varepsilon\geq\avol(\overline{D})$ for all $\varepsilon>0$ as desired.

(2): Let $\overline{\mathcal{R}}:=(\varphi:X'\to X;\overline{M})\in\widehat{\Theta}_{\rm ad}(\overline{D})$ and $\overline{\mathcal{S}}:=(\varphi:X'\to X;\overline{N})\in\widehat{\Theta}_{\rm ad}(\overline{E})$.
We have
\begin{align*}
 \avol(\overline{D}+\overline{E})&=\langle(\overline{D}+\overline{E})^{\cdot (d+1)}\rangle\geq \adeg\left((\overline{M}+\overline{N})^{\cdot (d+1)}\right)\\
 &=\sum_{i=0}^{d+1}\binom{d+1}{i}\adeg(\overline{M}^{\cdot i}\cdot\overline{N}^{\cdot (d-i+1)}).
\end{align*}
On the other hand, by using Proposition~\ref{prop:defaposint2}, we can see that
\[
 \sup_{\substack{\overline{\mathcal{R}}\in\widehat{\Theta}_{\rm ad}(\overline{D}) \\ \overline{\mathcal{S}}\in\widehat{\Theta}_{\rm ad}(\overline{E})}}\left\{\sum_{i=0}^{d+1}\binom{d+1}{i}\adeg(\overline{M}^{\cdot i}\cdot\overline{N}^{\cdot (d-i+1)})\right\}=\sum_{i=0}^{d+1}\binom{d+1}{i}\langle\overline{D}^{\cdot i}\cdot\overline{E}^{d-i+1}\rangle.
\]

(3): The first and the second inequalities follow from Theorem~\ref{thm:aKT} (2) and (3), respectively.
The last assertion follows from Proposition~\ref{prop:aposcont} (4).

(4): This follows from Theorem~\ref{thm:aKT} (4).
\end{proof}

\section{Limit expression}

In this section, we would like to give a limit expression for arithmetic positive intersection numbers (Proposition~\ref{prop:aposint}), which are closely related to the asymptotic intersection numbers of moving parts restricted to the strict transforms studied by Ein-Lazarsfeld-Musta\c{t}\u{a}-Nakamaye-Popa \cite[Definition~2.6]{Ein_Laz_Mus_Nak_Pop06}.
We shall use Proposition~\ref{prop:aposint} in a proof of Corollary~\ref{cor:avollimasymortho} but these results do not affect the main part of this paper, namely the proof of Theorems~A and B.

Suppose that $X$ is generically smooth, and let $\overline{D}$ be a big arithmetic $\QQ$-divisor of $C^0$-type.
For an integer $m\geq 1$ and for each function $P:\ZZ_{>0}\to\RR$ such that, for any $\delta>0$,
\[
 \exp(-m\delta)\leq P(m)\leq\exp(m\delta)
\]
holds for every sufficiently large $m\geq 1$ (for example, $P$ is a positive polynomial function), we construct a suitable birational morphism $\mu_m:X_m\to X$ and a decomposition $\mu_m^*(m\overline{D})=\overline{A}^P(m\overline{D})+\overline{B}^P(m\overline{D})$ into a sum of a ``moving part'' $\overline{A}^P(m\overline{D})$ and a ``fixed part'' $\overline{B}^P(m\overline{D})$.
In Proposition~\ref{prop:aposint}, we shall show that an arithmetic positive intersection number can be written as a limit of arithmetic intersection numbers with respect to the moving parts.

\begin{lemma}\label{lem:ampleext}
Let $M$ be a complex projective manifold of dimension $d$, and let $\overline{L}$ be a $C^{\infty}$ Hermitian holomorphic line bundle on $M$.
Suppose that $\overline{L}$ is positive.
Then, for any $\varepsilon>0$, there exists a positive integer $k_{\varepsilon}\geq 1$ such that, for any $k$ with $k\geq k_{\varepsilon}$ and for any $x\in M$, there exists a section $l^x\in\Hz(M,kL)$ such that
\[
 \|l^x\|_{\sup}^{k\overline{L}}\leq \exp(k\varepsilon)|l^x|_{k\overline{L}}(x).
\]
\end{lemma}

\begin{proof}
Let $\Phi_M$ be the normalized volume form associated to $c_1(\overline{L})$ and consider the $L^2$-norms, $\|\cdot\|_{L^2,\Phi_M}^{k\overline{L}}$, on $\Hz(M,kL)$.
By the Gromov inequality \cite[Theorem~3.4]{Yuan07}, one can compare $\|\cdot\|_{\sup}^{k\overline{L}}$ with $\|\cdot\|_{L^2,\Phi_M}^{k\overline{L}}$ as
\[
 \|\cdot\|_{L^2,\Phi_M}^{k\overline{L}}\leq\|\cdot\|_{\sup}^{k\overline{L}}\leq G(m+1)^d\|\cdot\|_{L^2,\Phi_M}^{k\overline{L}},
\]
where $G>0$ is a positive constant.
Denote $r_k:=\dim_{\CC}\Hz(M,kL)-1$.
Let $\phi_k:M\to P_k:=\PP(\Hz(M,kL))$ be a closed immersion associated to $|kL|$ for $k\gg 1$, and let $\mathcal{O}_{P_k}(1)$ be the hyperplane line bundle on $P_k$.
For each $k$, we fix an $L^2$-orthonormal basis for $\Hz(M,kL)$ with respect to $\|\cdot\|_{L^2,\Phi_M}^{k\overline{L}}$, and endow $\mathcal{O}_{P_k}(1)$ with the Fubini-Study metric induced from this basis.
We set $\overline{(kL)}^{\rm FS}:=\phi_k^*\overline{\mathcal{O}}_{P_k}^{\rm FS}(1)$.
Note that $\overline{\mathcal{O}}_{P_k}^{\rm FS}(1)$ is invariant under the special unitary group $SU(r_k+1)$.
By the theorem of Tian-Bouche (\cite[Theorem~A]{Tian90}, \cite[Th\'eor\`eme principal]{Bouche90}), $\log(|\cdot|_{k\overline{L}}/|\cdot|_{\overline{kL}^{\rm FS}})/k$ uniformly converges to 0 as $k\to\infty$.
There exists a $k_{\varepsilon}\geq 1$ such that, for every $k\geq k_{\varepsilon}$ and for every $x\in M$,
\begin{eqnarray}
 & G(m+1)^d\leq\exp(k\varepsilon/3), \\
 & \|\cdot\|_{\sup}^{k\overline{L}}\leq\exp(k\varepsilon/3)\|\cdot\|_{\sup}^{\overline{kL}^{\rm FS}},
\end{eqnarray}
and
\begin{equation}
 |\cdot|_{\overline{kL}^{\rm FS}}(x)\leq\exp(k\varepsilon/3)|\cdot|_{k\overline{L}}(x).
\end{equation}
Let $k\geq k_{\varepsilon}$ and fix a non-zero section $l_0\in\Hz(M,kL)$ and a closed point $x_0\in M$ such that the function $|l_0|_{\overline{kL}^{\rm FS}}$ attains its maximum at $x_0$, that is, $\|l_0\|_{\sup}^{\overline{kL}^{\rm FS}}=|l_0|_{\overline{kL}^{\rm FS}}(x_0)$.
Given any point $x\in M$, one can find a special unitary transform $g^x\in SU(r_k+1)$ such that $g^x(\phi_k(x))=\phi_k(x_0)$ and set $l:=l^x:=(g^x\circ\phi_k)^*l_0\in\Hz(X,kL)$.
Then we have $\|l\|_{L^2,\Phi_M}^{k\overline{L}}=\|l_0\|_{L^2,\Phi_M}^{k\overline{L}}$ and $|l|_{\overline{kL}^{\rm FS}}(x)=|l_0|_{\overline{kL}^{\rm FS}}(x_0)$.
All in all, we have
\begin{align*}
 \|l\|_{\sup}^{k\overline{L}} &\leq\exp(k\varepsilon/3)\|l\|_{L^2,\Phi_M}^{k\overline{L}}=\exp(k\varepsilon/3)\|l_0\|_{L^2,\Phi_M}^{k\overline{L}} \\
 &\leq\exp(k\varepsilon/3)\|l_0\|_{\sup}^{k\overline{L}}\leq\exp(2k\varepsilon/3)\|l_0\|_{\sup}^{\overline{kL}^{\rm FS}} \\
 &=\exp(2k\varepsilon/3)|l_0|_{\overline{kL}^{\rm FS}}(x_0)=\exp(2k\varepsilon/3)|l|_{\overline{kL}^{\rm FS}}(x)\leq\exp(k\varepsilon)|l|_{k\overline{L}}(x).
\end{align*}
\end{proof}

To obtain the limit expression, we use the method of distortion functions developed by Yuan \cite{Yuan07} and Moriwaki \cite{Moriwaki12a}.
Fix a normalized volume form $\Phi_X$ on $X(\CC)$.
For all $m\geq 1$, we consider the $L^2$-norms, $\|\cdot\|_{L^2,\Phi_X}^{m\overline{D}}$, with respect to $\Phi_X$ on $\Hz(X,mD)\otimes_{\ZZ}\CC$.
Let $r_m:=\rk\mathrm{F}^0(X,m\overline{D})-1$ and choose an $L^2$-orthonormal basis $(e_0,\dots,e_{r_m})$ for $\mathrm{F}^0(X,m\overline{D})\otimes_{\ZZ}\CC$.
The \emph{distortion function} with respect to $\mathrm{F}^0(X,m\overline{D})\otimes_{\ZZ}\CC$ is defined as
\begin{equation}
 \bm{\mathsf{B}}^0(m\overline{D})(x):=|e_0|_{m\overline{D}}^2(x)+\dots+|e_{r_m}|_{m\overline{D}}^2(x)
\end{equation}
for $x\in X(\CC)$, which does not depend on the choice of the $L^2$-orthonormal basis.

\begin{lemma}[\textup{\cite[Theorem~3.2.3]{Moriwaki12a}}]\label{lem:Moriwakidist}
There exists a positive constant $C>0$ having the following two properties:
\begin{enumerate}
\item[\textup{(i)}] $\bm{\mathsf{B}}^0(p\overline{D})(x)\leq C(p+1)^{3d}$ and
\item[\textup{(ii)}] $\displaystyle{\frac{\bm{\mathsf{B}}^0(p\overline{D})(x)}{C(p+1)^{3d}}\cdot\frac{\bm{\mathsf{B}}^0(q\overline{D})(x)}{C(q+1)^{3d}}\leq\frac{\bm{\mathsf{B}}^0((p+q)\overline{D})(x)}{C(p+q+1)^{3d}}}$
\end{enumerate}
for all $x\in X(\CC)$ and $p,q\geq 1$.
\end{lemma}

Suppose that $m\overline{D}\in\aBigCone(X;C^0)$.
Let $\mathfrak{b}^0(m\overline{D}):=\Image(\mathrm{F}^0(X,m\overline{D})\otimes_{\ZZ}\mathcal{O}_X(-mD)\to\mathcal{O}_X)$, and let $\mu_m:X_m\to X$ be a blowing up such that $X_m$ is generically smooth and normal and $\mu_m^{-1}\mathfrak{b}^0(m\overline{D})\cdot\mathcal{O}_{X_m}$ is Cartier.
Let $B(m\overline{D})$ be an effective Cartier divisor such that $\mathcal{O}_{X_m}(-B(m\overline{D}))=\mu_m^{-1}\mathfrak{b}^0(m\overline{D})\cdot\mathcal{O}_{X_m}$, and let $1_{B_m}$ be the canonical section.
Set $A(m\overline{D}):=m\mu_m^*D-B(m\overline{D})$.
Since the homomorphism
\[
 \mathrm{F}^0(X_m,m\mu_m^*\overline{D})\otimes_{\ZZ}\mathcal{O}_{X_m}(-m\mu_m^*D)\to\mathcal{O}_{X_m}(-B(m\overline{D}))
\]
is surjective, the homomorphism
\[
 \mathrm{F}^0(X_m,m\mu_m^*\overline{D})\otimes_{\ZZ}\mathcal{O}_{X_m}\to\mathcal{O}_{X_m}(A(m\overline{D}))
\]
is also surjective and we have an injective homomorphism $\mathrm{F}^0(X_m,m\mu_m^*\overline{D})\otimes_{\ZZ}\CC\to\Hz(X_m,A(m\overline{D}))\otimes_{\ZZ}\CC$ sending an $s\in\mathrm{F}^0(X_m,m\mu_m^*\overline{D})\otimes_{\ZZ}\CC$ to a section $\sigma\in\Hz(X_m,A(m\overline{D}))\otimes_{\ZZ}\CC$ such that $s=\sigma\otimes 1_{B_m}$.
For simplicity of notation, we shall sometimes identify $s\in\mathrm{F}^0(X_m,m\mu_m^*\overline{D})\otimes_{\ZZ}\CC$ with $\sigma\in\Hz(X_m,A(m\overline{D}))\otimes_{\ZZ}\CC$ if no confusion can arise.

\begin{lemma}\label{lem:metric}
Let $P(m)$ be a non-zero positive function such that $P(m)>0$ for all $m\geq 1$.
\begin{enumerate}
\item[\textup{(1)}] We can endow $\mathcal{O}_{X_m}(B(m\overline{D}))$ with a Hermitian metric defined by
\[
 |1_{B_m}|_{\overline{B}^P(m\overline{D})}(x):=\frac{\sqrt{\bm{\mathsf{B}}^0(m\overline{D})(\mu_m(x))}}{P(m)}
\]
for $x\in X_m(\CC)$.
Set $\overline{A}^P(m\overline{D}):=m\mu_m^*\overline{D}-\overline{B}^P(m\overline{D})$.
Then $\overline{A}^P(m\overline{D})\in\aDiv(X;C^{\infty})$ and the curvature form $\omega(\overline{A}^P(m\overline{D}))$ is semipositive.
\item[\textup{(2)}] Let $C$ be as in Lemma~\ref{lem:Moriwakidist}.
For any $\gamma,\gamma'\geq 0$ with
\[
 \exp(-m\gamma')\cdot C(m+1)^{3d}\leq P(m)^2\leq\exp(m\gamma),
\]
we have
\[
 (\mu_m;X_m\to X;\overline{A}^P(m\overline{D})/m+(0,\gamma))\in\widehat{\Theta}_{C^{\infty}}(\overline{D}+(0,\gamma+\gamma')).
\]
\end{enumerate}
\end{lemma}

\begin{proof}
(1): This follows from the same arguments as in Lemma~\ref{lem:filtered} (1).
In fact, if we choose an open covering $\{U_{\alpha}\}$ of $X_m(\CC)$ such that $\mu_m^*\mathcal{O}_{X}(mD)_{\CC}|_{U_{\alpha}}$ is trivial with local frame $\eta_{\alpha}$ and $B(m\overline{D})_{\CC}\cap U_{\alpha}$ is defined by a local equation $g_{\alpha}$, then we can write
\[
 \mu_m^*e_{i}=f_{\alpha,i}\cdot g_{\alpha}\cdot\eta_{\alpha}, \quad i=0,1,\dots,r_m
\]
on $U_{\alpha}$, where $f_{\alpha,0},\dots,f_{\alpha,r_m}$ are holomorphic functions on $U_{\alpha}$ satisfying $\{x\in U_{\alpha}\,|\,f_{\alpha,0}(x)=\dots=f_{\alpha,r_m}(x)=0\}=\emptyset$.
Since
\[
 \sqrt{\bm{\mathsf{B}}^0(m\overline{D})(\mu_m(x))}=|\eta_{\alpha}|_{m\mu_m^*\overline{D}}(x)\sqrt{|f_{\alpha,0}(x)|^2+\dots+|f_{\alpha,r_m}(x)|^2}\cdot |g_{\alpha}(x)|
\]
for $x\in U_{\alpha}$, we have the first half of (1).

For each point $x_0\in X_m(\CC)$, we find indices $\alpha,\iota$ with $x_0\in U_{\alpha}$ and $f_{\alpha,\iota}(x_0)\neq 0$.
Then
\[
 |\mu_m^*e_{\iota}|_{\overline{A}^P(m\overline{D})}^2(x)=\frac{|f_{\alpha,\iota}(x)|^2}{|f_{\alpha,0}(x)|^2+\dots+|f_{\alpha,r_m}(x)|^2}\cdot P(m)^2
\]
is a $C^{\infty}$-function on $U_{\alpha}$.
By reindexing, we may assume $\iota=0$.
Let $h_{\alpha,i}:=f_{\alpha,i}/f_{\alpha,0}$ near $x_0\in U_{\alpha}$.
Then
\begin{align*}
 \omega(\overline{A})&=\frac{\sqrt{-1}}{2\pi}\partial\overline{\partial}\log\left(1+|h_{\alpha,1}|^2+\dots+|h_{\alpha,r_m}|^2\right)\\
                            &=\frac{\sqrt{-1}}{2\pi}\left(\frac{1}{1+\sum_{i=1}^{r_m}|h_{\alpha,i}|^2}\sum_{j=1}^{r_m}dh_{\alpha,j}\wedge d\overline{h}_{\alpha,j}\right.\\
                            &\left.-\frac{1}{\left(1+\sum_{i=1}^{r_m}|h_{\alpha,i}|^2\right)^2}\left(\sum_{k=1}^{r_m}\overline{h}_{\alpha,k}dh_{\alpha,k}\right)\wedge\left(\sum_{l=1}^{r_m}h_{\alpha,l}d\overline{h}_{\alpha,l}\right)\right),
\end{align*}
is semipositive point-wise near $x_0\in U_{\alpha}$ since the Hermitian matrix
\begin{align*}
 &\frac{1}{1+\sum_{i=1}^{r_m}|h_{\alpha,i}|^2}\left(\begin{array}{ccc}1 & & \Largesymbol{O}\\ & \ddots & \\ \Largesymbol{O} & & 1\end{array}\right)\\
 &\qquad\qquad -\frac{1}{\left(1+\sum_{i=1}^{r_m}|h_{\alpha,i}|^2\right)^2}\left(\begin{array}{ccc}\overline{h}_{\alpha,1}h_{\alpha,1} & \cdots & \overline{h}_{\alpha,1}h_{\alpha,r_m}\\ \vdots & \ddots & \vdots\\ \overline{h}_{\alpha,r_m}h_{\alpha,1} & \cdots & \overline{h}_{\alpha,r_m}h_{\alpha,r_m}\end{array}\right)
\end{align*}
is positive-definite with eigenvalues $1/(1+\sum_i|h_{\alpha,i}|^2)^2,1/(1+\sum_i|h_{\alpha,i}|^2),\dots,1/(1+\sum_i|h_{\alpha,i}|^2)$.

(2): We have a decomposition
\[
 m\mu_m^*\overline{D}+(0,m(\gamma+\gamma'))=(\overline{A}^P(m\overline{D})+(0,m\gamma))+(\overline{B}^P(m\overline{D})+(0,m\gamma')).
\]
Since
\[
 |1_{B_m}|_{\overline{B}^P(m\overline{D})}(x):=\frac{\sqrt{\bm{\mathsf{B}}^0(m\overline{D})(\mu_m(x))}}{P(m)}\leq\exp(m\gamma'/2),
\]
$\overline{B}^P(m\overline{D})+(0,m\gamma')$ is effective.
Thus, it suffices to show that the homomorphism
\[
 \mathrm{F}^0(X_m,\overline{A}^P(m\overline{D})+(0,m\gamma))\otimes_{\ZZ}\mathcal{O}_{X_m}\to \mathcal{O}_{X_m}(A(m\overline{D}))
\]
is surjective.
Given $s\in\mathrm{F}^0(X,m\overline{D})$, we write $s=x_0e_0+\dots +x_{r_m}e_{r_m}$, $x_0,\dots,x_{r_m}\in\CC$.
Since by the Cauchy-Schwarz inequality
\begin{align*}
 |s|_{m\overline{D}}(\mu_m(x))&\leq |x_0||e_0|_{m\overline{D}}(\mu_m(x))+\dots +|x_{r_m}||e_{r_m}|_{m\overline{D}}(\mu_m(x))\\
 &\leq \|s\|_{L^2,\Phi_X}^{m\overline{D}}\times\sqrt{\bm{\mathsf{B}}^0(m\overline{D})(\mu_m(x))}\\
 &\leq \|s\|_{\sup}^{m\overline{D}}\times \sqrt{\bm{\mathsf{B}}^0(m\overline{D})(\mu_m(x))}
\end{align*}
for $x\in X_m(\CC)$, we have $\|\mu_m^*s\|_{\sup}^{\overline{A}^P(m\overline{D})}\leq P(m)\|s\|_{\sup}^{m\overline{D}}\leq\exp(m\gamma/2)\|s\|_{\sup}^{m\overline{D}}$.
Since $\mathrm{F}^0(X_m,m\mu_m^*\overline{D})\otimes_{\ZZ}\mathcal{O}_{X_m}\to \mathcal{O}_{X_m}(A(m\overline{D}))$ is surjective and
\[
\xymatrix{\mathrm{F}^0(X_m,\overline{A}^P(m\overline{D})+(0,m\gamma))\otimes_{\ZZ}\mathcal{O}_{X_m} \ar[r] & \mathcal{O}_{X_m}(A(m\overline{D})) \\
 \mathrm{F}^0(X_m,m\mu_m^*\overline{D})\otimes_{\ZZ}\mathcal{O}_{X_m} \ar[u] \ar[r] & \mathcal{O}_{X_m}(A(m\overline{D})) \ar@{=}[u]
}
\]
is commutative, we conclude the proof.
\end{proof}

Let $\overline{D}\in\aBigCone_{\QQ}(X;C^0)$, $m\geq 1$ an integer such that $m\overline{D}\in\aBigCone(X;C^0)$, and $P(m)$ a non-zero positive function such that $P(m)>0$ for all $m\geq 1$ and, given any $\delta>0$, we have
\[
 \exp(-m\delta)\leq P(m)\leq\exp(m\delta)
\]
for all $m\gg 1$.

\begin{proposition}\label{prop:aposint}
Suppose that $X$ is generically smooth and that $\overline{D}\in\aBigCone_{\QQ}(X;C^0)$.
Let $\overline{D}_k,\dots,\overline{D}_n\in\aBigCone_{\RR}(X;C^0)$ and $\overline{D}_{n+1},\dots,\overline{D}_d\in\aDiv_{\RR}^{\rm Nef}(X;C^0)$.
Then the arithmetic positive intersection number of $(\overbrace{\overline{D},\dots,\overline{D}}^{k},\overline{D}_k,\dots,\overline{D}_n;\overline{D}_{n+1},\dots,\overline{D}_d)$ can be represented as a limit:
\[
 \langle\overline{D}^{\cdot k}\cdot\overline{D}_k\cdots\overline{D}_n\rangle\overline{D}_{n+1}\cdots\overline{D}_d=\lim_{m\to\infty}\frac{\langle\mu_m^*\overline{D}_k\cdots\mu_m^*\overline{D}_n\rangle\overline{A}^P(m\overline{D})^{\cdot k}\cdot\mu_m^*\overline{D}_{n+1}\cdots\mu_m^*\overline{D}_d}{m^k},
\]
where the limit is taken over all $m\geq 1$ with $m\overline{D}\in\aBigCone(X;C^0)$.
\end{proposition}

\begin{proof}
Let $C>0$ be as in Lemma~\ref{lem:Moriwakidist}.
We may concentrate on the case $P(m):=\sqrt{C(m+1)^{3d}}$ since the general case easily follows from this case.
We set $\overline{A}_m:=\overline{A}^P(m\overline{D})$ and $\overline{B}_m:=\overline{B}^P(m\overline{D})$ for simplicity.
By the multilinearity in the variables $\overline{D}_{n+1},\dots,\overline{D}_d$, we may assume without loss of generality that $\overline{D}_{n+1},\dots,\overline{D}_d$ are all nef and big.
Set $S:=\{m\geq 1\,|\,m\overline{D}\in\aBigCone(X;C^0)\}$, and
\[
 I_m:=\langle\mu_m^*\overline{D}_k\cdots\mu_m^*\overline{D}_n\rangle\cdot\overline{A}_m^{\cdot k}\cdot\mu_m^*\overline{D}_{n+1}\cdots\mu_m^*\overline{D}_d
\]
for $m\in S$.
Note that $S$ is naturally a sub-semigroup of $\NN$.
For $p,q\in S$, let $\mu_{p,q}:X_{p,q}\to X$ be a blowing up such that $X_{p,q}$ is generically smooth and normal and $\mu_{p,q}$ factors as $X_{p,q}\xrightarrow{\nu_m}X_m\xrightarrow{\mu_m}X$ for $m=p,q,p+q$.
Since $\nu_p^*1_{B_p}\otimes \nu_q^*1_{B_q}$ vanishes along $\mu_{p,q}^{-1}\Bs\mathrm{F}^0(X,(p+q)\overline{D})$, there exists a section $1_{p,q}\in\Hz(X_{p,q},\nu_p^*B_p+\nu_q^*B_q-\nu_{p+q}^*B_{p+q})$ such that $1_{p,q}\otimes \nu_{p+q}^*1_{B_{p+q}}=\nu_p^*1_{B_p}\otimes \nu_q^*1_{B_q}$ and that
\begin{multline*}
 \|1_{p,q}\|_{\sup}^{\nu_p^*\overline{B}_p+\nu_q^*\overline{B}_q-\nu_{p+q}^*\overline{B}_{p+q}}\\
 =\sup_{x\in X_{p,q}(\CC)}\sqrt{\frac{\bm{\mathsf{B}}^0(p\overline{D})(\mu_{p,q}(x))}{C(p+1)^{3d}}}\cdot\sqrt{\frac{\bm{\mathsf{B}}^0(q\overline{D})(\mu_{p,q}(x))}{C(q+1)^{3d}}}\cdot\sqrt{\frac{C(p+q+1)^{3d}}{\bm{\mathsf{B}}^0((p+q)\overline{D})(\mu_{p,q}(x))}}\leq 1
\end{multline*}
by Lemma~\ref{lem:Moriwakidist}.
Hence, we have $\nu_{p+q}^*\overline{A}_{p+q}\geq\nu_p^*\overline{A}_p+\nu_q^*\overline{A}_q$.
By Lemmas~\ref{lem:aint2} (3) and Proposition~\ref{prop:avollogconcave} (4), we have
\[
 I_{p+q}^{1/k}\geq I_p^{1/k}+I_q^{1/k}
\]
for all $p,q\in S$, which implies that the sequence $(I_m^{1/k}/m)_{m\in S}$ converges.

Let $\varepsilon>0$ be an arbitrarily small positive real number and fix a real number $\delta>0$ such that
\begin{equation}\label{eqn:deltachoice}
 |\langle(\overline{D}+(0,\delta))^{\cdot k}\cdot\overline{D}_{k}\cdots\overline{D}_n\rangle\overline{D}_{n+1}\cdots\overline{D}_d-\langle\overline{D}^{\cdot k}\cdot\overline{D}_k\cdots\overline{D}_n\rangle\overline{D}_{n+1}\cdots\overline{D}_d|\leq\varepsilon.
\end{equation}
Let $m_{\delta}\geq 1$ be a positive integer such that $\exp(-m\delta)\cdot C(m+1)^{3d}\leq 1$ for all $m\geq m_{\delta}$.
Then $(\mu_m:X_m\to X,\overline{A}_m/m+(0,\delta))\in\widehat{\Theta}_{C^{\infty}}(\overline{D}+(0,\delta))$ for all $m\geq m_{\delta}$ and we have
\begin{align*}
 &\langle\overline{D}^{\cdot k}\cdot\overline{D}_k\cdots\overline{D}_n\rangle\overline{D}_{n+1}\cdots\overline{D}_d+\varepsilon \\
 &\qquad\qquad\geq\langle(\overline{D}+(0,\delta))^{\cdot k}\cdot\overline{D}_k\cdots\overline{D}_n\rangle\overline{D}_{n+1}\cdots\overline{D}_d\\
 &\qquad\qquad\geq\langle\mu_m^*\overline{D}_k\cdots\mu_m^*\overline{D}_n\rangle(\overline{A}_m/m+(0,\delta))^{\cdot k}\cdot\mu_m^*\overline{D}_{n+1}\cdots\mu_m^*\overline{D}_d \\
 &\qquad\qquad\geq I_m/m
\end{align*}
for all $m\geq m_{\delta}$.
Hence $\langle\overline{D}^{\cdot k}\cdot\overline{D}_k\cdots\overline{D}_n\rangle\overline{D}_{n+1}\cdots\overline{D}_d\geq\lim_{m\to\infty}I_m/m$.

By Proposition~\ref{prop:replacewithample}, we can fix admissible approximations $\overline{\mathcal{R}}:=(\varphi:X'\to X;\overline{M})\in\widehat{\Theta}_{\rm amp}(\overline{D})$ and $\overline{\mathcal{R}}_i:=(\varphi:X'\to X;\overline{M}_i)\in\widehat{\Theta}_{\rm ad}(\overline{D}_i)$ for $i=k,\dots,n$ such that
\begin{equation}\label{eqn:posintchoice}
 \overline{\mathcal{R}}^{\cdot k}\cdot\overline{\mathcal{R}}_k\cdots\overline{\mathcal{R}}_n\cdot\overline{D}_{n+1}\cdots\overline{D}_d\geq \langle\overline{D}^{\cdot k}\cdot\overline{D}_k\cdots\overline{D}_n\rangle\overline{D}_{n+1}\cdots\overline{D}_d-\varepsilon>\varepsilon.
\end{equation}
Note that, since $\overline{D}\in\aBigCone_{\QQ}(X;C^0)$ and $\overline{F}:=\varphi^*\overline{D}-\overline{M}\in\aDiv_{\QQ}(X;C^0)$, $\overline{M}$ is automatically an ample arithmetic $\QQ$-divisor.
Let $\gamma>0$ be a sufficiently small real number such that $\overline{M}-(0,\gamma)$ is still ample and
\begin{equation}\label{eqn:gammachoice}
 \overline{\mathcal{R}}_k\cdots\overline{\mathcal{R}}_n\cdot(\overline{M}-(0,\gamma))^{\cdot k}\cdot\varphi^*\overline{D}_{n+1}\cdots\varphi^*\overline{D}_d\geq\overline{\mathcal{R}}^{\cdot k}\cdot\overline{\mathcal{R}}_k\cdots\overline{\mathcal{R}}_n\cdot\overline{D}_{n+1}\cdots\overline{D}_d-\varepsilon.
\end{equation}
Fix a sufficiently divisible positive integer $m\in S$ having the properties that
\begin{itemize}
\item $m\overline{D}\in\aBigCone(X;C^0)$,
\item $\mathcal{O}_{X'}(mM)$ is a very ample line bundle,
\item $\mathrm{F}^{0+}(X',m\overline{M})=\Hz(X',mM)$,
\item for any $x\in X'(\CC)$, there exists a non-zero section $l\in\Hz(X',mM)\otimes_{\ZZ}\CC$ such that $\|l\|_{\sup}^{m\overline{M}}\leq\exp(m\gamma/2)|l|_{m\overline{M}}(x)$ (Lemma~\ref{lem:ampleext}),
\item $\mathcal{O}_{X'}(m\overline{F})$ is an effective continuous Hermitian line bundle, and
\item $C(m+1)^{3d}\leq\exp(m\gamma)$.
\end{itemize}
Fix a non-zero section $s\in\Hz(X',mF)$ having supremum norm less than or equal to one.
Let $\pi:X''\to X$ be a blowing up such that $X''$ is generically smooth and normal and that $\pi$ factors as $X''\xrightarrow{\psi}X'\xrightarrow{\varphi}X$ and as $X''\xrightarrow{\nu_m}X_m\xrightarrow{\mu_m}X$.
Since $\mathrm{F}^{0+}(X',m\overline{M})\otimes_{\ZZ}\mathcal{O}_{X'}\to mM$ is surjective, $s$ vanishes along $\varphi^{-1}\Bs\mathrm{F}^0(X,m\overline{D})$ and there exists a section $\sigma\in\Hz(X'',m\psi^*F-\nu_m^*B_m)$ such that $\sigma\otimes \nu_m^*1_{B_m}=\psi^*s$.

\begin{claim}\label{clm:normofsigma}
\[
 \exp(-m\gamma)\|\sigma\|_{\sup}^{m\psi^*\overline{F}-\nu_m^*\overline{B}_m}\leq 1.
\]
In particular, $\nu_m^*\overline{A}_m\geq m\psi^*(\overline{M}-(0,\gamma))$.
\end{claim}

\begin{proof}
Given any closed point $x\in X''(\CC)$, we can choose a non-zero section $l\in\Hz(X',mM)\otimes_{\ZZ}\CC$ such that
\begin{equation}\label{eqn:normofl}
 \|l\|_{\sup}^{m\overline{M}}\leq \exp(m\gamma/2)|l|_{m\overline{M}}(\psi(x)).
\end{equation}
Then,
\begin{align}
 |\sigma|_{m\psi^*\overline{F}-\nu_m^*\overline{B}_m}(x) \label{eqn:normofsigma}&=|s|_{m\overline{F}}(\psi(x))\cdot\sqrt{\frac{C(m+1)^{3d}}{\bm{\mathsf{B}}^0(m\overline{D})(\pi(x))}}\\
 &=\frac{|s\otimes l|_{m\varphi^*\overline{D}}(\psi(x))}{\sqrt{\bm{\mathsf{B}}^0(m\overline{D})(\pi(x))}}\times\frac{\sqrt{C(m+1)^{3d}}}{|l|_{m\overline{M}}(\psi(x))}.\nonumber
\end{align}
Since $\Hz(X',mM)=\mathrm{F}^0(X',m\overline{M})$, we can regard $s\otimes l\in\mathrm{F}^0(X,m\overline{D})\otimes_{\ZZ}\CC$.
Thus, by the Cauchy-Schwarz inequality, we have
\begin{align}
 |s\otimes l|_{m\varphi^*\overline{D}}(\psi(x))&=|s\otimes l|_{m\overline{D}}(\pi(x))\label{eqn:CS}\\
 &\leq\|l\|_{\sup}^{m\overline{M}}\times\sqrt{\bm{\mathsf{B}}^0(m\overline{D})(\pi(x))}. \nonumber
\end{align}
By combining (\ref{eqn:normofl}), (\ref{eqn:normofsigma}), and (\ref{eqn:CS}), we have
\[
 |\sigma|_{m\psi^*\overline{F}-\nu_m^*\overline{B}_m}(x)\leq\frac{\|l\|_{\sup}^{m\overline{M}}}{|l|_{m\overline{M}}(\psi(x))}\times \sqrt{C(m+1)^{3d}}\leq\exp(m\gamma).
\]
for every $x\in X''(\CC)$.
\end{proof}

By (\ref{eqn:posintchoice}), (\ref{eqn:gammachoice}), Claim~\ref{clm:normofsigma}, and Lemma~\ref{lem:aint2} (3), we have
\[
 I_m/m^k\geq\langle\overline{D}^{\cdot k}\cdot\overline{D}_k\cdots\overline{D}_n\rangle\overline{D}_{n+1}\cdots\overline{D}_d-2\varepsilon
\]
for all sufficiently divisible $m\gg 1$.
\end{proof}

\section{Differentiability of the arithmetic volumes}\label{sec:PMT}

Let $X$ be a normal projective arithmetic variety, and let $\overline{D}$ and $\overline{E}$ be two arithmetic $\RR$-divisors on $X$.
In this section, we show that the function $\RR\ni t\mapsto\avol(\overline{D}+t\overline{E})\in\RR$ is differentiable provided that $\overline{D}$ is big.
By the arithmetic Siu inequality \cite[Theorem~1.2]{Yuan07} and the continuity of the arithmetic volume function, we have
\begin{equation}\label{eqn:aSiu}
 \avol(\overline{D}-\overline{E})\geq\adeg(\overline{D}^{\cdot (d+1)})-(d+1)\adeg(\overline{D}^{\cdot d}\cdot\overline{E})
\end{equation}
if both $\overline{D}$ and $\overline{E}$ are nef.

\begin{proposition}\label{prop:aSiu}
Let $\overline{D}$ and $\overline{E}$ be two arithmetic $\RR$-divisors on $X$ and suppose that $\overline{D}$ is nef.
\begin{enumerate}
\item[\textup{(1)}] Suppose that there exists a nef and big arithmetic $\RR$-divisor $\overline{A}$ such that $\overline{A}\pm\overline{E}$ is nef and $\overline{A}-\overline{D}$ is pseudo-effective.
Set $C_1(|t|):=2d(d+1)(1+|t|)^{d-1}$.
Then
\[
 \avol(\overline{D}+t\overline{E})-\avol(\overline{D})\geq (d+1)\adeg(\overline{D}^{\cdot d}\cdot\overline{E})\cdot t-C_1(|t|)\avol(\overline{A})\cdot t^2
\]
for all $t\in\RR$.
\item[\textup{(2)}] Suppose that $\overline{E}$ is pseudo-effective and that there exists a nef and big arithmetic $\RR$-divisor $\overline{A}$ such that $\overline{A}+(\overline{D}+\overline{E})$ is nef and $\overline{A}-(\overline{D}+\overline{E})$ is pseudo-effective.
Set $C_2(t):=4d(d+1)(1+2t)^{d-1}$.
Then
\[
 \avol(\overline{D}+t\overline{E})-\avol(\overline{D})\geq (d+1)\adeg(\overline{D}^{\cdot d}\cdot\overline{E})\cdot t-C_2(t)\avol(\overline{A})\cdot t^2
\]
for all $t\in\RR_{\geq 0}$.
\end{enumerate}
\end{proposition}

\begin{remark}\label{rem:integrablenefbig}
If $\overline{E}$ is integrable, then we can write $\overline{E}=\overline{M}-\overline{N}$ with nef and big arithmetic $\RR$-divisors $\overline{M},\overline{N}$.
Set $\overline{A}:=\overline{D}+\overline{M}+\overline{N}$.
Then $\overline{A}\pm\overline{E}$ and $\overline{A}-\overline{D}$ are all nef and big, and the condition of Proposition~\ref{prop:aSiu} (1) is satisfied.
Similarly, if $\overline{D}+\overline{E}$ is integrable, then one can find an $\overline{A}$ satisfying the condition of Proposition~\ref{prop:aSiu} (2).
\end{remark}

\begin{proof}
(1): If $t=0$, then the assertion is trivial.
For $t\in\RR\setminus\{0\}$, we write $\sgn(t):=t/|t|$ and set $\overline{B}:=\overline{A}-\sgn(t)\overline{E}$.
Since $\overline{D}$, $\overline{A}$, and $\overline{B}$ are all nef, we have
\begin{align}
 \avol(\overline{D}+t\overline{E})& =\avol((\overline{D}+|t|\overline{A})-|t|\overline{B}) \label{eqn:aSiu1}\\
 &\geq\adeg((\overline{D}+|t|\overline{A})^{\cdot (d+1)})-(d+1)\adeg((\overline{D}+|t|\overline{A})^{\cdot d}\cdot |t|\overline{B}) \nonumber\\
 &\geq\adeg(\overline{D}^{\cdot (d+1)})+(d+1)\adeg(\overline{D}^{\cdot d}\cdot |t|\overline{A}) \nonumber\\
 &\qquad\qquad\qquad\qquad -(d+1)\adeg((\overline{D}+|t|\overline{A})^{\cdot d}\cdot |t|\overline{B}) \nonumber
\end{align}
by (\ref{eqn:aSiu}).
Moreover, since $\overline{A}-\overline{D}$ and $2\overline{A}-\overline{B}=\overline{A}+\sgn(t)\overline{E}$ are pseudo-effective, we have
\begin{align}
 \adeg((\overline{D}+|t|\overline{A})^{\cdot d}\cdot |t|\overline{B}) &=\sum_{k=0}^{d}\binom{d}{k}\adeg(\overline{D}^{\cdot (d-k)}\cdot\overline{A}^{\cdot k}\cdot\overline{B})\cdot |t|^{k+1} \label{eqn:aSiu2}\\
 & \leq\adeg(\overline{D}^{\cdot d}\cdot |t|\overline{B})+2\avol(\overline{A})\sum_{k=1}^{d}\binom{d}{k}|t|^{k+1}. \nonumber
\end{align}
By (\ref{eqn:aSiu1}), (\ref{eqn:aSiu2}), and $|t|(\overline{A}-\overline{B})=t\overline{E}$, we have
\[
 \avol(\overline{D}+t\overline{E})-\avol(\overline{D})\geq (d+1)\adeg(\overline{D}^{\cdot d}\cdot\overline{E})\cdot t-C(|t|)\avol(\overline{A})\cdot t^2,
\]
where
\[
 C(|t|):=2(d+1)\sum_{k=1}^{d}\binom{d}{k}|t|^{k-1}\leq 2d(d+1)(1+|t|)^{d-1}.
\]

(2): The proof is almost the same as the above.
Set $\overline{B}:=\overline{A}+\overline{D}+\overline{E}$.
Since $\overline{D}$, $\overline{A}$, and $\overline{B}$ are all nef, we have
\begin{align}
 \avol(\overline{D}+t\overline{E}) &=\avol((\overline{D}+t\overline{B})-t(\overline{A}+\overline{D})) \label{eqn:aSiu3}\\
 &\geq\adeg(\overline{D}^{\cdot (d+1)})+(d+1)\adeg(\overline{D}^{\cdot d}\cdot t\overline{B}) \nonumber\\
 &\qquad\qquad\qquad\qquad -(d+1)\adeg((\overline{D}+t\overline{B})^{\cdot d}\cdot t(\overline{A}+\overline{D})) \nonumber
\end{align}
by using (\ref{eqn:aSiu}).
Since $\overline{A}-\overline{D}$ and $2\overline{A}-\overline{B}=\overline{A}-\overline{D}-\overline{E}$ are pseudo-effective, we have
\begin{equation}\label{eqn:aSiu4}
 \adeg((\overline{D}+t\overline{B})^{\cdot d}\cdot t(\overline{A}+\overline{D}))\leq\adeg(\overline{D}^{\cdot d}\cdot t(\overline{A}+\overline{D}))+\avol(\overline{A})\sum_{k=1}^{d}\binom{d}{k}(2t)^{k+1}.
\end{equation}
Hence, by (\ref{eqn:aSiu3}), (\ref{eqn:aSiu4}), we have
\[
 \avol(\overline{D}+t\overline{E})-\avol(\overline{D})\geq (d+1)\adeg(\overline{D}^{\cdot d}\cdot\overline{E})\cdot t-C'(t)\avol(\overline{A})\cdot t^2,
\]
where
\[
 C'(t):=4(d+1)\sum_{k=1}^{d}\binom{d}{k}(2t)^{k-1}\leq 4d(d+1)(1+2t)^{d-1}.
\]
\end{proof}

\begin{theorem}\label{thm:diffavol}
For any $\overline{D}\in\aBigCone_{\RR}(X;C^0)$ and $\overline{E}\in\aDiv_{\RR}(X;C^0)$, the function
\[
 \RR\ni t\mapsto\avol(\overline{D}+t\overline{E})\in\RR
\]
is differentiable, and
\[
 \lim_{t\to 0}\frac{\avol(\overline{D}+t\overline{E})-\avol(\overline{D})}{t}=(d+1)\langle\overline{D}^{\cdot d}\rangle\overline{E}.
\]
\end{theorem}

\begin{proof}
First, we suppose that $\overline{E}$ is integrable, and fix a nef and big arithmetic $\RR$-divisor $\overline{A}$ such that $\overline{A}\pm\overline{E}$ is nef and $\overline{A}-\overline{D}$ is pseudo-effective (see Remark~\ref{rem:integrablenefbig}).
Set $C:=2^dd(d+1)$.
Then by Proposition~\ref{prop:aSiu} (1), for any $t\in\RR$ with $|t|\leq 1$ and for any $(\varphi;\overline{M})\in\widehat{\Theta}(\overline{D})$,
\[
 \avol(\overline{D}+t\overline{E})\geq\avol(\overline{M}+t\varphi^*\overline{E})\geq \avol(\overline{M})+(d+1)\adeg(\overline{M}^{\cdot d}\cdot\varphi^*\overline{E})\cdot t-C\avol(\overline{A})\cdot t^2
\]
and, for any $t\in\RR$ with $|t|\leq 1$ and for any $(\varphi_t;\overline{M}_t)\in\widehat{\Theta}(\overline{D}+t\overline{E})$,
\[
 \avol(\overline{D})\geq\avol(\overline{M}_t-t\overline{E})\geq \avol(\overline{M}_t)-(d+1)\adeg(\overline{M}_t^{\cdot d}\cdot\varphi_t^*\overline{E})\cdot t-C\avol(2\overline{A})\cdot t^2.
\]
Since $\overline{D}+t\overline{E}$ is big for all $t$ with $|t|$ sufficiently small, we have
\begin{equation}
 \avol(\overline{D}+t\overline{E})-\avol(\overline{D})\geq (d+1)t\langle\overline{D}^{\cdot d}\rangle\overline{E}-Ct^2\avol(\overline{A})
\end{equation}
and
\begin{equation}
 \avol(\overline{D})-\avol(\overline{D}+t\overline{E})\geq -(d+1)t\langle(\overline{D}+t\overline{E})^{\cdot d}\rangle\overline{E}-Ct^2\avol(2\overline{A})
\end{equation}
for all $t$ with $|t|\ll 1$ by using Proposition~\ref{prop:avollogconcave} (1).
Thus, by Proposition~\ref{prop:aposcont} (2), we conclude the proof in this case.

Next in general, we can assume that $X$ is generically smooth.
By the Stone-Weierstrass theorem, we can find a sequence of continuous functions $(f_n)_{n\geq 1}$ such that $\overline{E}+(0,2f_n)$ is $C^{\infty}$ and $\|f_n\|_{\sup}\to 0$ as $n\to\infty$.
Since
\begin{multline*}
 \left|\frac{\avol(\overline{D}+t\overline{E})-\avol(\overline{D})}{t}-\frac{\avol(\overline{D}+t(\overline{E}+(0,2f_n)))-\avol(\overline{D})}{t}\right|\\
 \leq (d+1)\|f_n\|_{\sup}\vol(D_{\QQ}+tE_{\QQ})
\end{multline*}
for all $t\in\RR\setminus\{0\}$ and $n\geq 1$, the function $\RR\ni t\mapsto\avol(\overline{D}+t\overline{E})\in\RR$ is differentiable at $t=0$ and
\[
 \lim_{t\to 0}\frac{\avol(\overline{D}+t\overline{E})-\avol(\overline{D})}{t}=(d+1)\langle\overline{D}^{\cdot d}\rangle\overline{E}
\]
by Proposition~\ref{prop:aposcont} (3).
\end{proof}

\begin{corollary}\label{cor:asymorthoapos}
For $\overline{D}\in\aBigCone_{\RR}(X;C^0)$, we have $\avol(\overline{D})=\langle\overline{D}^{\cdot d}\rangle\overline{D}$.
\end{corollary}

\begin{proof}
This is clear since $\avol((1+t)\overline{D})=(1+t)^{d+1}\avol(\overline{D})$.
\end{proof}

Corollary~\ref{cor:asymorthoapos} can be regarded as a version of the asymptotic orthogonality of the approximate Zariski decompositions.
In particular, we can show that the decompositions $m\mu_m^*\overline{D}=\overline{A}^P(m\overline{D})+\overline{B}^P(m\overline{D})$ given in Proposition~\ref{prop:aposint} is asymptotically orthogonal.
Moriwaki \cite[Theorem~9.3.5]{Moriwaki12a} proved a similar result when $\dim X$ is two.

\begin{corollary}\label{cor:avollimasymortho}
Let $\overline{D}$ be a big arithmetic $\QQ$-divisor, and let $\mu_m^*(m\overline{D})=\overline{A}^P(m\overline{D})+\overline{B}^P(m\overline{D})$ be as in Proposition~\ref{prop:aposint}.
Then we have
\[
 \avol(\overline{D})=\lim_{m\to\infty}\frac{\adeg(\overline{A}^P(m\overline{D})^{\cdot (d+1)})}{m^{d+1}}\quad\text{and}\quad\lim_{m\to\infty}\frac{\adeg(\overline{A}^P(m\overline{D})^{\cdot d}\cdot\overline{B}^P(m\overline{D}))}{m^{d+1}}=0,
\]
where the limit is taken over all $m\geq 1$ with $m\overline{D}\in\aBigCone(X;C^0)$.
\end{corollary}

\begin{proof}
What we have to show is
\[
 \lim_{m\to\infty}\frac{\adeg(\overline{A}^P(m\overline{D})^{\cdot d}\cdot\mu_m^*\overline{E})}{m^d}=\langle\overline{D}^{\cdot d}\rangle\overline{E}
\]
for every $\overline{E}$ on $X$.
This is true when $\overline{E}$ is integrable (see Proposition~\ref{prop:aposint}) and, in general, we can approximate $\overline{E}$ by arithmetic $\RR$-divisors of $C^{\infty}$-type.
\end{proof}

In the rest of this section, we would like to apply Theorem~\ref{thm:diffavol} to the problem of the equidistribution of rational points on $X$ (see \cite{Yuan07, Berman_Boucksom, Chen11}).
For $\overline{D}\in\aBigCone_{\RR}(X;C^0)$, we set
\[
 h_{\overline{D}}^+(X):=\frac{\avol(\overline{D})}{(d+1)\vol(D_{\QQ})}.
\]
A sequence $(x_n)_{n\geq 1}$ of rational points on $X$ is called \emph{generic} if for any closed subscheme $Y\subseteq X$, $x_n\notin Y(\overline{\QQ})$ holds for every $n\gg 1$.

\begin{lemma}\label{lem:height}
Let $\overline{D}=a_1\overline{D}_1+\dots+a_l\overline{D}_l$ be a big arithmetic $\RR$-divisor on $X$, where $a_i>0$ and $\overline{D}_i$ is a big arithmetic divisor.
\begin{enumerate}
\item[\textup{(1)}] Suppose that $D_i$ are effective, and let $x\in X(\overline{\QQ})$ be a rational point such that $x\notin\Supp(D_i)$ for all $i$.
Then we have $h_{\overline{D}}(x)\geq 0$.
\item[\textup{(2)}] Let $(x_n)_{n\geq 1}$ be a generic sequence of rational points on $X$.
Then
\[
 \liminf_{n\to\infty}h_{\overline{D}}(x_n)\geq h_{\overline{D}}^+(X).
\]
\end{enumerate}
\end{lemma}

\begin{proof}
(1): Let $C_x$ be the arithmetic curve corresponding to $x$.
Since $x\notin\Supp(D_i)$ for all $i$, we have
\[
 h_{\overline{D}}(x)=\frac{1}{[K(x):\QQ]}\left(\sum_{i=1}^la_i\log\sharp\left(\mathcal{O}_{C_x}(D_i)/\mathcal{O}_{C_x}\right)+\frac{1}{2}\sum_{\sigma:K(x)\to\CC}g_{\overline{D}}(x^{\sigma})\right)\geq 0.
\]

(2): For any $\lambda\in\RR$ with $\avol(\overline{D}-(0,2\lambda))>0$, we have $h_{\overline{D}-(0,2\lambda)}(x_n)\geq 0$ for all $n\gg 1$.
Thus
\begin{equation}\label{eqn:height1}
 \liminf_{n\to\infty}h_{\overline{D}}(x_n)\geq\lambda.
\end{equation}
On the other hand, for any $\lambda\in\RR$ with $\avol(\overline{D}-(0,2\lambda))=0$, we have
\begin{equation}\label{eqn:height2}
 \lambda\geq\frac{\avol(\overline{D})}{(d+1)\vol(D_{\QQ})}
\end{equation}
by Lemma~\ref{lem:defineachivol}.
Hence, by (\ref{eqn:height1}) and (\ref{eqn:height2}), we have
\[
 \liminf_{n\to\infty}h_{\overline{D}}(x_n)\geq\sup\{\lambda\in\RR\,|\,\avol(\overline{D}-(0,2\lambda))>0\}\geq\frac{\avol(\overline{D})}{(d+1)\vol(D_{\QQ})}.
\]
\end{proof}

\begin{corollary}\label{cor:equidistthm}
\begin{enumerate}
\item[\textup{(1)}] For $\overline{D}\in\aBigCone_{\RR}(X;C^0)$ and for $f\in C^0(X)$, we have
\[
 \lim_{t\to 0}\frac{h^+_{\overline{D}+t(0,f)}(X)-h^+_{\overline{D}}(X)}{t}=\frac{\langle\overline{D}^{\cdot d}\rangle(0,f)}{\vol(D_{\QQ})}.
\]
\item[\textup{(2)}] Let $(x_n)_{n\geq 1}$ be a generic sequence of rational points on $X$, and let $\overline{D}$ be a big arithmetic $\RR$-divisor on $X$.
If $h_{\overline{D}}(x_n)$ converges to $h_{\overline{D}}^+(X)$, then, for any $f\in C^0(X)$,
\[
 \lim_{n\to\infty}\frac{1}{[K(x_n):\QQ]}\sum_{\sigma:K(x_n)\to\CC}f(x_n^{\sigma})=\frac{\langle\overline{D}^{\cdot d}\rangle(0,2f)}{\vol(D_{\QQ})}.
\]
\end{enumerate}
\end{corollary}

\begin{proof}
(1) follows from Theorem~\ref{thm:diffavol}.

(2): Note that
\[
 h_{(0,2f)}(x_n)=\frac{1}{[K(x_n):\QQ]}\sum_{\sigma:K(x_n)\to\CC}f(x_n^{\sigma}).
\]
and
\[
 \liminf_{n\to\infty}h_{\overline{D}+t(0,2f)}(x_n)\geq h_{\overline{D}+t(0,2f)}^+(X)
\]
for all $t$ with $|t|\ll1$ (Lemma~\ref{lem:height}).
Since $h_{\overline{D}}(x_n)\to h_{\overline{D}}(X)$ as $n\to\infty$, we have
\begin{align*}
 \liminf_{n\to\infty}h_{(0,2f)}(x_n) &=\frac{\liminf_{n\to\infty}h_{\overline{D}+t(0,2f)}(x_n)-\lim_{n\to\infty}h_{\overline{D}}(x_n)}{t} \\
 &\geq\frac{h_{\overline{D}+t(0,2f)}^+(X)-h_{\overline{D}}^+(X)}{t}
\end{align*}
for $t>0$ and
\begin{align*}
 \limsup_{n\to\infty}h_{(0,2f)}(x_n) &=\frac{\liminf_{n\to\infty}h_{\overline{D}+t(0,2f)}(x_n)-\lim_{n\to\infty}h_{\overline{D}}(x_n)}{t} \\
 &\leq\frac{h_{\overline{D}+t(0,2f)}^+(X)-h_{\overline{D}}^+(X)}{t}
\end{align*}
for $t<0$.
Thus the sequence $\left(h_{(0,2f)}(x_n)\right)_{n\geq 1}$ converges and we conclude the proof.
\end{proof}

\begin{remark}
We can see from the proof of Corollary~\ref{cor:equidistthm} that the function
\[
 \RR\ni t\mapsto\liminf_{n\to\infty}h_{\overline{D}+t(0,2f)}(x_n)\in\RR
\]
is differentiable at $t=0$ with the same derivative as in Corollary~\ref{cor:equidistthm} (2).
\end{remark}

\section{A criterion for the pseudo-effectivity}

The goal of this section is to give a numerical characterization of the pseudo-effectivity of arithmetic $\RR$-divisors (Theorem~\ref{thm:psef}).
Our arguments are based on Boucksom-Demailly-Paun-Peternell \cite{Boucksom_Demailly_Paun_Peternell} and uses the generalized Dirichlet unit theorem of Moriwaki \cite{Moriwaki10b}.
Let $X$ be a normal projective arithmetic variety of dimension $d+1$, and let $\overline{D}$ be a big arithmetic $\RR$-divisor on $X$.
To begin with, we give an explicit estimate for the asymptotic orthogonality of admissible approximations under the assumption that $\overline{D}$ is \emph{integrable}.

\begin{proposition}\label{prop:asymortho}
Suppose that $\overline{D}$ is integrable and fix a nef and big arithmetic $\RR$-divisor $\overline{A}$ such that $\overline{A}\pm\overline{D}$ is nef and big.
Then
\[
 \adeg(\overline{M}^{\cdot d}\cdot\overline{F})^2\leq 20\avol(\overline{A})\cdot (\avol(\overline{D})-\avol(\overline{M}))
\]
for any birational morphism of normal projective arithmetic varieties $\varphi:X'\to X$, and for any decomposition $\varphi^*\overline{D}=\overline{M}+\overline{F}$ such that $\overline{M}$ is a nef arithmetic $\RR$-divisor on $X'$ and $\overline{F}$ is a pseudo-effective arithmetic $\RR$-divisor on $X'$.
\end{proposition}

\begin{proof}
Applying Proposition~\ref{prop:aSiu} (2) to $\overline{M}+t\overline{F}$, we have
\begin{align*}
 &\avol(\overline{D})\geq\avol(\overline{M}+t\overline{F})\\
 &\quad\geq\avol(\overline{M})+(d+1)\adeg(\overline{M}^{\cdot d}\cdot\overline{F})\cdot t-4d(d+1)(1+2t)^{d-1}\avol(\overline{A})\cdot t^2
\end{align*}
for $t\geq 0$.
Set
\[
 0<t=\frac{\adeg(\overline{M}^{\cdot d}\cdot\overline{F})}{10(d+1)\avol(\overline{A})}\leq \frac{1}{10(d+1)}.
\]
Since $(1+2t)^{d-1}\leq \left(1+\frac{1}{5(d+1)}\right)^{d-1}\leq\exp(\frac{1}{5})\leq \frac{5}{4}$, we have
\[
 \avol(\overline{D})\geq\avol(\overline{M})+\frac{\adeg(\overline{M}^{\cdot d}\cdot\overline{F})^2}{20\avol(\overline{A})}.
\]
\end{proof}

Recall that we can uniquely extend the arithmetic intersection product to a continuous multilinear map
\[
 \aDiv_{\RR}(X;C^0)\times\aDiv^{\rm Nef}_{\RR}(X;C^0)^{\times d}\to\RR,\quad (\overline{D}_0;\overline{D}_1,\dots,\overline{D}_d)\mapsto\adeg(\overline{D}_0\cdots\overline{D}_d),
\]
having the property that, if $\overline{D}_0$ is pseudo-effective and $\overline{D}_1,\dots,\overline{D}_d$ are nef, then
\[
 \adeg(\overline{D}_0\cdots\overline{D}_d)\geq 0
\]
(Lemma~\ref{lem:aintkiso}).

\begin{lemma}\label{lem:equalitycond}
\begin{enumerate}
\item[\textup{(1)}] Let $\overline{D}\in\aDiv_{\RR}(X;C^0)$, and let $\overline{H}_1,\dots,\overline{H}_d$ be ample arithmetic $\RR$-divisors on $X$.
If $\overline{D}\geq 0$, then
\[
 \adeg(\overline{D}\cdot\overline{H}_1\cdots\overline{H}_d)\geq 0.
\]
The equality holds if and only if $\overline{D}=0$.
\item[\textup{(2)}] Let $\phi\in\Rat(X)^{\times}\otimes_{\ZZ}\RR$.
If $\widehat{(\phi)}\geq 0$, then $\widehat{(\phi)}=0$.
\end{enumerate}
\end{lemma}

\begin{proof}
(1): We write $\overline{D}=(D,g_{\overline{D}})$ and $D=\sum_{i=1}^l a_iD_i$, where $a_i\geq 0$ and $D_i$ is an effective prime divisor.
Suppose that the equality holds.
Note that, since $\overline{H}_1,\dots,\overline{H}_d$ are ample, we can restrict them to $D_i$.
Since
\begin{multline*}
 \adeg(\overline{D}\cdot\overline{H}_1\cdots\overline{H}_d) \\
 =\sum_{i=1}^la_i\adeg(\overline{H}_1|_{D_i}\cdots\overline{H}_d|_{D_i})+\frac{1}{2}\int_{X(\CC)}g_{\overline{D}}\,\omega(\overline{H}_1)\wedge\dots\wedge\omega(\overline{H}_d)=0,
\end{multline*}
we have $a_1=\dots=a_l=0$ and $g_{\overline{D}}\equiv 0$.

(2): Let $\overline{H}$ be an ample arithmetic divisor on $X$.
By the linearity in the last variable, $\adeg(\overline{H}^{\cdot d}\cdot\widehat{(\phi)})=0$ holds.
Thus (2) follows from (1).
\end{proof}

\begin{remark}
One can see that a $\phi\in\Rat(X)^{\times}\otimes_{\ZZ}\RR$ satisfies $\widehat{(\phi)}=0$ if and only if $\phi\in\Hz(X,\mathcal{O}_X^*)\otimes_{\ZZ}\RR\subseteq\Rat(X)^{\times}\otimes_{\ZZ}\RR$.
\end{remark}

\begin{theorem}\label{thm:psef}
Let $X$ be a normal projective arithmetic variety, and let $\overline{D}$ be an arithmetic $\RR$-divisor on $X$.
(We do not assume that $\overline{D}$ is integrable.)
\begin{enumerate}
\item[\textup{(1)}] The following are equivalent.
\begin{enumerate}
\item[\textup{(i)}] $\overline{D}$ is pseudo-effective.
\item[\textup{(ii)}] For any normalized blow-up $\varphi:X'\to X$ and for any nef arithmetic $\RR$-divisor $\overline{H}$ on $X'$, we have
\[
 \adeg(\varphi^*\overline{D}\cdot\overline{H}^{\cdot d})\geq 0.
\]
\item[\textup{(iii)}] For any blowing up $\varphi:X'\to X$ such that $X'$ is generically smooth and normal and for any ample arithmetic $\QQ$-divisor $\overline{H}$ on $X'$, we have
\[
 \adeg(\varphi^*\overline{D}\cdot\overline{H}^{\cdot d})\geq 0.
\]
\end{enumerate}
\item[\textup{(2)}] Suppose that $\overline{D}$ is pseudo-effective.
The following are equivalent.
\begin{enumerate}
\item[\textup{(i)}] There exists a $\phi\in\Rat(X)^{\times}\otimes_{\ZZ}\RR$ such that $\overline{D}=\widehat{(\phi)}$.
\item[\textup{(ii)}] There exist a blowing up $\varphi:X'\to X$ such that $X'$ is generically smooth and normal and an ample arithmetic $\RR$-divisor $\overline{H}$ on $X'$ such that
\[
 \adeg(\varphi^*\overline{D}\cdot\overline{H}^{\cdot d})=0.
\]
\end{enumerate}
\end{enumerate}
\end{theorem}

\begin{proof}
(1): (i) $\Rightarrow$ (ii) and (ii) $\Rightarrow$ (iii) are clear.

(iii) $\Rightarrow$ (i): First, we assume that $\overline{D}$ is integrable and fix a nef and big arithmetic $\QQ$-divisor $\overline{A}$ on $X$ such that $\overline{A}\pm\overline{D}$ is nef and big.
Set $\sigma:=-s(\overline{D},\overline{A}):=-\sup\{t\in\RR\,|\,\text{$\overline{D}-t\overline{A}$ is pseudo-effective}\}$.
If $\sigma\leq 0$, then $\overline{D}$ is pseudo-effective, so that we can assume $\sigma>0$ and try to deduce a contradiction from it.
Set $\overline{D}':=\overline{D}+\sigma\overline{A}$.
Then, for any blowing up $\varphi:Y\to X$ such that $Y$ is generically smooth and normal and for any ample arithmetic $\QQ$-divisor $\overline{H}$ on $Y$, we have
\begin{equation}\label{eqn:psef0}
 \adeg(\varphi^*\overline{D}'\cdot\overline{H}^{\cdot d})=\adeg(\varphi^*\overline{D}\cdot\overline{H}^{\cdot d})+\sigma\adeg(\varphi^*\overline{A}\cdot\overline{H}^{\cdot d})\geq\sigma\adeg(\varphi^*\overline{A}\cdot\overline{H}^{\cdot d}).
\end{equation}
Note that $\overline{D}'$ is pseudo-effective, integrable, and $\avol(\overline{D}')=0$.
Thus $\overline{D}'+\varepsilon\overline{A}$ is big and integrable for every $\varepsilon$ with $0<\varepsilon<1$.
By applying the arithmetic Fujita approximation to $\overline{D}'+\varepsilon\overline{A}$, one can find a blow-up $\varphi:X'\to X$ such that $X'$ is generically smooth and normal and a decomposition
\begin{equation}\label{eqn:psef1}
 \varphi^*(\overline{D}'+\varepsilon\overline{A})=\overline{M}+(\overline{F}+\widehat{(\phi)})
\end{equation}
such that $\overline{M}$ is an ample arithmetic $\QQ$-divisor, $\overline{F}$ is an effective arithmetic $\RR$-divisor, $\phi\in\Rat(X')^{\times}\otimes_{\ZZ}\RR$, and
\begin{equation}\label{eqn:psef2}
 \frac{1}{2}\varepsilon^{d+1}\avol(\overline{A})\leq\avol(\overline{M})\leq\avol(\overline{D}'+\varepsilon\overline{A})\leq\avol(\overline{M})+\varepsilon^{2(d+1)}
\end{equation}
(see Proposition~\ref{prop:replacewithample}).
Since $(\sigma+2)\overline{A}\pm(\overline{D}'+\varepsilon\overline{A})=(\overline{A}\pm\overline{D})+((\sigma+1)\pm(\sigma+\varepsilon)\overline{A})$ is nef and big, we can apply Proposition~\ref{prop:asymortho} to the decomposition (\ref{eqn:psef1}) and obtain
\begin{equation}\label{eqn:psef3}
 \adeg(\overline{M}^{\cdot d}\cdot\overline{F})^2\leq 20(\sigma+2)^{d+1}\avol(\overline{A})\varepsilon^{2(d+1)}.
\end{equation}
Moreover, by Theorem~\ref{thm:aKT} (2), we have
\begin{equation}\label{eqn:psef4}
 \adeg(\varphi^*\overline{A}\cdot\overline{M}^{\cdot d})\geq\avol(\overline{A})^{\frac{1}{d+1}}\cdot\avol(\overline{M})^{\frac{d}{d+1}}.
\end{equation}
Hence, by (\ref{eqn:psef0}), (\ref{eqn:psef2}), (\ref{eqn:psef3}), and (\ref{eqn:psef4}), we have
\begin{align*}
 0<\sigma &\leq \frac{\adeg(\varphi^*\overline{D}'\cdot\overline{M}^{\cdot d})}{\adeg(\varphi^*\overline{A}\cdot\overline{M}^{\cdot d})}\leq \frac{\adeg(\varphi^*(\overline{D}'+\varepsilon\overline{A})\cdot\overline{M}^{\cdot d})}{\avol(\overline{A})^{\frac{1}{d+1}}\cdot\avol(\overline{M})^{\frac{d}{d+1}}}=\frac{\avol(\overline{M})+\adeg(\overline{M}^{\cdot d}\cdot\overline{F})}{\avol(\overline{A})^{\frac{1}{d+1}}\cdot\avol(\overline{M})^{\frac{d}{d+1}}} \\
 &\leq\left(\frac{\avol(\overline{D}'+\varepsilon\overline{A})}{\avol(\overline{A})}\right)^{\frac{1}{d+1}}+\varepsilon\cdot 2^{\frac{d}{d+1}}\cdot\left(\frac{20(\sigma+2)^{d+1}}{\avol(\overline{A})}\right)^{\frac{1}{2}}.
\end{align*}
This leads us to a contradiction since the right-hand-side tends to zero as $\varepsilon\to 0$.

Next, we consider the general case.
We assume that $X$ is generically smooth and choose a sequence of non-negative continuous functions $(f_n)_{n\geq 1}$ such that $\overline{D}+(0,f_n)$ is $C^{\infty}$ and $\|f_n\|_{\sup}\to 0$ as $n\to\infty$.
Since
\[
 \adeg(\varphi^*(\overline{D}+(0,f_n))\cdot\overline{H}^{\cdot d})\geq 0
\]
for any blow-up $\varphi:Y\to X$ such that $Y$ is generically smooth and normal and for any ample arithmetic $\QQ$-divisor $\overline{H}$ on $Y$, $\overline{D}+(0,f_n)$ is pseudo-effective for every $n$.
Thus, for every big arithmetic $\RR$-divisor $\overline{B}$ on $X$, we have
\[
 \avol(\overline{D}+\overline{B}+(0,f_n))\geq\avol(\overline{B})>0.
\]
This implies that $\overline{D}$ is pseudo-effective.

(2): Since (i) $\Rightarrow$ (ii) is obvious, we are going to show (ii) $\Rightarrow$ (i).
First we show that for any arithmetic $\RR$-divisors of $C^{\infty}$-type, $\overline{D}_1,\dots,\overline{D}_d$, on $X'$ we have
\begin{equation}\label{eqn:psefequal}
 \adeg(\varphi^*\overline{D}\cdot\overline{D}_1\cdots\overline{D}_d)=0.
\end{equation}
Suppose that $\overline{H}_1,\dots,\overline{H}_d$ are all ample.
One can find an $\alpha\gg 0$ such that $\alpha\overline{H}-\overline{H}_i$ is nef and big for every $i$.
Since
\[
 0\leq\adeg(\varphi^*\overline{D}\cdot\overline{H}_1\cdots\overline{H}_d)\leq\adeg(\varphi^*\overline{D}\cdot(\alpha\overline{H})\cdots\overline{H}_d)\leq\cdots\leq\alpha^d\adeg(\varphi^*\overline{D}\cdot\overline{H}^{\cdot d})=0,
\]
we have $\adeg(\varphi^*\overline{D}\cdot\overline{H}_1\cdots\overline{H}_d)=0$.
Since each $\overline{D}_i$ can be written as a difference of two ample arithmetic $\RR$-divisors, we have (\ref{eqn:psefequal}).
Hence, in particular,
\[
 \deg(\varphi^*D_{\QQ}\cdot H_{\QQ}^{\cdot (d-1)})=\adeg(\varphi^*\overline{D}\cdot(0,2)\cdot\overline{H}^{\cdot (d-1)})=0.
\]
Therefore, $\varphi^*D_{\QQ}$ is numerically trivial on $X_{\QQ}'$ and one can apply the generalized Dirichlet theorem of Moriwaki \cite{Moriwaki10b} to $\overline{D}$.
There exists a $\phi\in\Rat(X')^{\times}\otimes_{\ZZ}\RR$ such that $\varphi^*\overline{D}-\widehat{(\phi)}$ is effective.
Thus by Lemma~\ref{lem:equalitycond} (1), we have $\varphi^*\overline{D}=\widehat{(\phi)}$.
This descends to $X$ since $X$ is normal.
\end{proof}

\section{Concavity of the arithmetic volumes}

In this section, we obtain an arithmetic version of the Discant inequality (Theorem~\ref{thm:discant}) and prove that the arithmetic volume function is strictly concave over the cone of nef and big arithmetic $\RR$-divisors (Theorem~\ref{thm:strictconcave}).
As applications, we give some numerical characterizations of the Zariski decompositions (Corollary~\ref{cor:Zariski} and Proposition~\ref{prop:Zariski}).

\begin{theorem}[An arithmetic Discant inequality]\label{thm:discant}
Let $X$ be a normal projective arithmetic variety of dimension $d+1$, and let $\overline{D}$ and $\overline{P}$ be two big arithmetic $\RR$-divisors on $X$.
If $\overline{P}$ is nef, then we have
\[
 0\leq\left(\left(\langle\overline{D}^{\cdot d}\rangle\overline{P}\right)^{\frac{1}{d}}-s\avol(\overline{P})^{\frac{1}{d}}\right)^{d+1}\leq\left(\langle\overline{D}^{\cdot d}\rangle\overline{P}\right)^{1+\frac{1}{d}}-\avol(\overline{D})\avol(\overline{P})^{\frac{1}{d}},
\]
where $s=s(\overline{D},\overline{P}):=\sup\{t\in\RR\,|\,\text{$\overline{D}-t\overline{P}$ is pseudo-effective}\}$.
\end{theorem}

\begin{proof}
Since $\avol(\overline{D}-t\overline{P})>0$ for $t<s$ and $\avol(\overline{D}-s\overline{P})=0$, we have
\begin{equation}\label{eqn:discant1}
 \avol(\overline{D})=(d+1)\int_{t=0}^s\langle(\overline{D}-t\overline{P})^{\cdot d}\rangle\overline{P}\,dt
\end{equation}
by Theorem~\ref{thm:diffavol}.
On the other hand,
\begin{equation}\label{eqn:discant2}
 0\leq\langle(\overline{D}-t\overline{P})^{\cdot d}\rangle\overline{P}\leq\left(\left(\langle\overline{D}^{\cdot d}\rangle\overline{P}\right)^{\frac{1}{d}}-t\avol(\overline{P})^{\frac{1}{d}}\right)^d
\end{equation}
for all $t<s$ by Proposition~\ref{prop:avollogconcave} (4).
By (\ref{eqn:discant1}) and (\ref{eqn:discant2}), we have
\begin{align*}
 \avol(\overline{D})\avol(\overline{P})^{\frac{1}{d}} &\leq (d+1)\avol(\overline{P})^{\frac{1}{d}}\int_{t=0}^s\left(\left(\langle\overline{D}^{\cdot d}\rangle\overline{P}\right)^{\frac{1}{d}}-t\avol(\overline{P})^{\frac{1}{d}}\right)^d\,dt\\
 &=\left(\langle\overline{D}^{\cdot d}\rangle\overline{P}\right)^{1+\frac{1}{d}}-\left(\left(\langle\overline{D}^{\cdot d}\rangle\overline{P}\right)^{\frac{1}{d}}-s\avol(\overline{P})^{\frac{1}{d}}\right)^{d+1}
\end{align*}
as desired.
\end{proof}

\begin{remark}
Let $\overline{D}$ and $\overline{E}$ be two big arithmetic $\RR$-divisors on $X$.
By the same arguments as above, we can prove
\[
 0\leq\left(\langle\overline{D}^{\cdot d}\cdot\overline{E}\rangle^{\frac{1}{d}}-s'\avol(\overline{E})^{\frac{1}{d}}\right)^{d+1}\leq\langle\overline{D}^{\cdot d}\cdot\overline{E}\rangle^{1+\frac{1}{d}}-\avol(\overline{D})\avol(\overline{E})^{\frac{1}{d}},
\]
where we set $s':=\inf_{(\varphi;\overline{M})\in\widehat{\Theta}(\overline{E})}s(\overline{D},\overline{M})\geq s(\overline{D},\overline{E})$.
If $\overline{E}$ is not nef, then $s'>s$ in general.
\end{remark}

\begin{corollary}\label{cor:Teissier}
\begin{enumerate}
\item[\textup{(1)}] Let $\overline{D},\overline{P}$ be big arithmetic $\RR$-divisors.
If $\overline{P}$ is nef, then
\[
 \frac{\left(\langle\overline{D}^{\cdot d}\rangle\overline{P}\right)^{\frac{1}{d}}-\left(\left(\langle\overline{D}^{\cdot d}\rangle\overline{P}\right)^{1+\frac{1}{d}}-\avol(\overline{D})\avol(\overline{P})^{\frac{1}{d}}\right)^{\frac{1}{d+1}}}{\avol(\overline{P})^{\frac{1}{d}}}\leq s(\overline{D},\overline{P})\leq\frac{\avol(\overline{D})}{\langle\overline{D}^{\cdot d}\rangle\overline{P}}.
\]
\item[\textup{(2)}] Suppose that $d=1$.
Let $\overline{D},\overline{E}$ be nef and big arithmetic $\RR$-divisors.
Then
\[
 \frac{\avol(\overline{E})^2}{4}\left(\frac{1}{s(\overline{E},\overline{D})}-s(\overline{D},\overline{E})\right)^2\leq\adeg(\overline{D}\cdot\overline{E})^2-\avol(\overline{D})\avol(\overline{E}).
\]
\end{enumerate}
\end{corollary}

\begin{proof}
(1): Since $\overline{D}-s(\overline{D},\overline{P})\overline{P}$ is pseudo-effective, we have
\[
 0<s(\overline{D},\overline{P})\langle\overline{D}^{\cdot d}\rangle\overline{P}\leq\avol(\overline{D})\quad\text{and}\quad s(\overline{D},\overline{P})\avol(\overline{P})^{\frac{1}{d}}\leq\left(\langle\overline{D}^{\cdot d}\rangle\overline{P}\right)^{\frac{1}{d}}
\]
by Proposition~\ref{prop:aposcont} (1).
Thus by Theorem~\ref{thm:discant}, we have the result.

(2): Since the left-hand-side of (1) is positive, we have
\begin{equation}\label{eqn:teissier}
 \frac{\adeg(\overline{D}\cdot\overline{E})}{\avol(\overline{E})}\leq\frac{1}{s(\overline{E},\overline{D})}\leq\frac{\avol(\overline{D})}{\adeg(\overline{D}\cdot\overline{E})-\left(\adeg(\overline{D}\cdot\overline{E})^2-\avol(\overline{D})\avol(\overline{E})\right)^{\frac{1}{2}}}.
\end{equation}
Since $\avol(\overline{D})\avol(\overline{E})\leq\adeg(\overline{D}\cdot\overline{E})^2$ by Theorem~\ref{thm:aKT} (1), we have
\[
 \frac{\avol(\overline{E})^2}{4}\left(\frac{1}{s(\overline{E},\overline{D})}-s(\overline{D},\overline{E})\right)^2\leq\adeg(\overline{D}\cdot\overline{E})^2-\avol(\overline{D})\avol(\overline{E})
\]
by (1) and (\ref{eqn:teissier}).
\end{proof}

\begin{theorem}\label{thm:strictconcave}
Let $X$ be a normal projective arithmetic variety of dimension $d+1$, and let $\overline{D}$ and $\overline{E}$ be two nef and big arithmetic $\RR$-divisors on $X$.
The following four conditions are equivalent.
\begin{enumerate}
\item[\textup{(1)}] $\avol(\overline{D}+\overline{E})^{\frac{1}{d+1}}=\avol(\overline{D})^{\frac{1}{d+1}}+\avol(\overline{E})^{\frac{1}{d+1}}$.
\item[\textup{(2)}] The function $i\mapsto\log\adeg(\overline{D}^{\cdot i}\cdot\overline{E}^{\cdot (d-i+1)})$ is affine: that is, for any $i$ with $1\leq i\leq d$, we have $\adeg(\overline{D}^{\cdot i}\cdot\overline{E}^{\cdot (d-i+1)})=\avol(\overline{D})^{\frac{i}{d+1}}\cdot\avol(\overline{E})^{\frac{d-i+1}{d+1}}$.
\item[\textup{(3)}] $\adeg(\overline{D}^{\cdot d}\cdot\overline{E})=\avol(\overline{D})^{\frac{d}{d+1}}\cdot\avol(\overline{E})^{\frac{1}{d+1}}$.
\item[\textup{(4)}] There exists a $\phi\in\Rat(X)^{\times}\otimes_{\ZZ}\RR$ such that
\[
 \frac{\overline{D}}{\avol(\overline{D})^{\frac{1}{d+1}}}=\frac{\overline{E}}{\avol(\overline{E})^{\frac{1}{d+1}}}+\widehat{(\phi)}.
\]
\end{enumerate}
\end{theorem}

\begin{proof}
(2) $\Rightarrow$ (3) and (4) $\Rightarrow$ (1) are clear.

(1) $\Rightarrow$ (2) follows from Proposition~\ref{prop:avollogconcave} (2), (3).

We prove (3) $\Rightarrow$ (2) by induction on $i$.
The case where $i=d$ is nothing but (3).
Suppose that the assertion holds for $i$.
Since
\begin{align*}
 \avol(\overline{D})^{\frac{i}{d+1}}\cdot\avol(\overline{E})^{\frac{d-i+1}{d+1}}&=\adeg(\overline{D}^{\cdot i}\cdot\overline{E}^{\cdot (d-i+1)})\\
 &\geq\adeg(\overline{D}^{\cdot (i-1)}\cdot\overline{E}^{\cdot (d-i+2)})^{\frac{1}{2}}\cdot\adeg(\overline{D}^{\cdot (i+1)}\cdot\overline{E}^{\cdot (d-i)})^{\frac{1}{2}}\\
 &\geq\avol(\overline{D})^{\frac{i}{d+1}}\cdot\avol(\overline{E})^{\frac{d-i+1}{d+1}},
\end{align*}
we have $\adeg(\overline{D}^{\cdot (i-1)}\cdot\overline{E}^{\cdot (d-i+2)})=\avol(\overline{D})^{\frac{i-1}{d+1}}\cdot\avol(\overline{E})^{\frac{d-i+2}{d+1}}$.

(2) $\Rightarrow$ (4): By applying Theorem~\ref{thm:discant} to $\overline{D}$ and $\overline{E}$, we have
\[
 s=s(\overline{D},\overline{E})=\left(\frac{\avol(\overline{D})}{\avol(\overline{E})}\right)^{\frac{1}{d+1}}\quad\text{and}\quad s(\overline{E},\overline{D})=\left(\frac{\avol(\overline{E})}{\avol(\overline{D})}\right)^{\frac{1}{d+1}}=s^{-1}.
\]
Let $\varphi:X'\to X$ be a blow-up such that $X'$ is generically smooth and normal, and let $\overline{H}$ be an ample arithmetic divisor on $X'$.
Since both $\varphi^*\overline{D}-s\varphi^*\overline{E}$ and $s\varphi^*\overline{E}-\varphi^*\overline{D}$ are pseudo-effective, we have
\[
 \adeg((\varphi^*\overline{D}-s\varphi^*\overline{E})\cdot\overline{H}^{\cdot d})=0.
\]
Thus, by Theorem~\ref{thm:psef} (2), there exists a $\phi\in\Rat(X)^{\times}\otimes_{\ZZ}\RR$ such that $\overline{D}-s\overline{E}=\widehat{(\phi)}$.
\end{proof}

In Corollaries~\ref{cor:equalnefbig} and \ref{cor:Zariski}, we generalize Moriwaki's results \cite[Corollary~4.2.2]{Moriwaki12b} for arithmetic surfaces to arithmetic varieties of arbitrary dimension.

\begin{corollary}\label{cor:equalnefbig}
Let $\overline{P}$ and $\overline{Q}$ be two nef and big arithmetic $\RR$-divisors.
Suppose that $\avol(\overline{P})=\avol(\overline{Q})$.
\begin{enumerate}
\item[\textup{(1)}] If $\overline{Q}-\overline{P}$ is pseudo-effective, then there exists a $\phi\in\Rat(X)^{\times}\otimes_{\ZZ}\RR$ such that $\overline{Q}-\overline{P}=\widehat{(\phi)}$.
\item[\textup{(2)}] If $\overline{Q}-\overline{P}$ is effective, then $\overline{P}=\overline{Q}$.
\end{enumerate}
\end{corollary}

\begin{proof}
(1): Since $\avol(2\overline{P})\leq\avol(\overline{P}+\overline{Q})\leq\avol(2\overline{Q})$, we have
\[
 \avol(\overline{P}+\overline{Q})^{\frac{1}{d+1}}=\avol(\overline{P})^{\frac{1}{d+1}}+\avol(\overline{Q})^{\frac{1}{d+1}}.
\]
Thus by Theorem~\ref{thm:strictconcave}, there exists a $\phi\in\Rat(X)^{\times}\otimes_{\ZZ}\RR$ such that $\overline{Q}-\overline{P}=\widehat{(\phi)}$.

(2): This follows from (1) and Lemma~\ref{lem:equalitycond} (2).
\end{proof}

Let $X$ be a normal and generically smooth projective arithmetic variety, and let $\overline{D}$ be a big arithmetic $\RR$-divisor on $X$.
A \emph{Zariski decomposition} of $\overline{D}$ is a decomposition $\overline{D}=\overline{P}+\overline{N}$ such that
\begin{enumerate}
\item[\textup{(1)}] $\overline{P}$ is a nef arithmetic $\RR$-divisor,
\item[\textup{(2)}] $\overline{N}$ is an effective arithmetic $\RR$-divisor, and
\item[\textup{(3)}] $\avol(\overline{P})=\avol(\overline{D})$
\end{enumerate}
(see also \cite[\S 4]{Moriwaki12b}).

\begin{corollary}\label{cor:Zariski}
The Zariski decomposition of $\overline{D}$ (if it exists) is unique: that is, if $\overline{D}=\overline{P}'+\overline{N}'$ is another Zariski decomposition of $\overline{D}$, then $\overline{P}=\overline{P}'$ and $\overline{N}=\overline{N}'$.
\end{corollary}

\begin{proof}
Since
\[
 2\avol(\overline{D})^{\frac{1}{d+1}}=\avol(\overline{P})^{\frac{1}{d+1}}+\avol(\overline{P}')^{\frac{1}{d+1}}\leq\avol(\overline{P}+\overline{P}')^{\frac{1}{d+1}}\leq\avol(2\overline{D})^{\frac{1}{d+1}}
\]
by the Brunn-Minkowski inequality, there exists a $\phi\in\Rat(X)^{\times}\otimes_{\ZZ}\RR$ such that $\overline{P}'=\overline{P}-\widehat{(\phi)}$ and $\overline{N}'=\overline{N}+\widehat{(\phi)}$ by Theorem~\ref{thm:strictconcave}.
On the other hand, since
\[
 \mult_x(N)=\mult_x(N')
\]
for all $x\in X_{\QQ}$ by \cite[Theorem~4.1.1]{Moriwaki12b}, we have $\mult_x(\phi)=0$ for all $x\in X_{\QQ}$.
Thus $\widehat{(\phi)}=0$.
\end{proof}

Lastly, we relate the Zariski decomposition of $\overline{D}$ in the above sense with arithmetic positive intersection numbers.

\begin{proposition}\label{prop:Zariski}
Let $\overline{D}$ be a big arithmetic $\RR$-divisor, and let $\overline{D}=\overline{P}+\overline{N}$ be a decomposition such that $\overline{P}$ is nef and $\overline{N}$ is effective.
The following two conditions are equivalent.
\begin{enumerate}
\item[\textup{(1)}] $\overline{D}=\overline{P}+\overline{N}$ is a Zariski decomposition of $\overline{D}$ in the above sense: that is, $\avol(\overline{P})=\avol(\overline{D})$.
\item[\textup{(2)}] For any integers $k,n$ with $0\leq k\leq n\leq d$, for any $\overline{D}_k,\dots,\overline{D}_n\in\aBigCone_{\RR}(X;C^0)$ and for any $\overline{D}_{n+1},\dots,\overline{D}_d\in\aDiv_{\RR}^{\rm Nef}(X;C^0)$, we have
\[
 \langle\overline{D}^{\cdot k}\cdot\overline{D}_k\cdots\overline{D}_n\rangle\overline{D}_{n+1}\cdots\overline{D}_d=\langle\overline{D}_k\cdots\overline{D}_n\rangle\overline{P}^{\cdot k}\cdot\overline{D}_{n+1}\cdots\overline{D}_d.
\]
\end{enumerate}
\end{proposition}

\begin{proof}
(2) $\Rightarrow$ (1) is clear since $\avol(\overline{D})=\langle\overline{D}^{\cdot (d+1)}\rangle=\avol(\overline{P})$.

(1) $\Rightarrow$ (2): We may assume that $\overline{D}_{n+1},\dots,\overline{D}_d$ are all nef.
The inequality
\[
 \langle\overline{D}^{\cdot k}\cdot\overline{D}_k\cdots\overline{D}_n\rangle\overline{D}_{n+1}\cdots\overline{D}_d\geq\langle\overline{D}_k\cdots\overline{D}_n\rangle\overline{P}^{\cdot k}\cdot\overline{D}_{n+1}\cdots\overline{D}_d
\]
is clear.
By blowing up the irreducible components of $\Supp(N)$, we can assume that $\overline{N}=a_1\overline{N}_1+\cdots+a_l\overline{N}_l$, where $a_1,\dots,a_l\in\RR_{>0}$ and $\overline{N}_1,\dots,\overline{N}_l$ are effective arithmetic divisors (see \cite[Proposition~2.4.2]{Moriwaki12a} for the existence of a decomposition of $g_{\overline{N}}$).
Let $\varepsilon>0$.
First, we choose an effective arithmetic $\QQ$-divisor $\overline{N}'$ such that $\overline{N}'\leq\overline{N}$ and
\begin{equation}\label{eqn:zariski1}
 \langle(\overline{P}+\overline{N}')^{\cdot k}\cdot\overline{D}_k\cdots\overline{D}_n\rangle\overline{D}_{n+1}\cdots\overline{D}_d+\varepsilon\geq\langle\overline{D}^{\cdot k}\cdot\overline{D}_k\cdots\overline{D}_n\rangle\overline{D}_{n+1}\cdots\overline{D}_d.
\end{equation}
We set $\overline{D}':=\overline{P}+\overline{N}'$.
Since $\overline{P}\leq\overline{D}'\leq\overline{D}$, we have $\avol(\overline{D}')=\avol(\overline{P})$.
Next, we choose $(\varphi;\overline{M})\in\widehat{\Theta}_{\rm ad}(\overline{D}')$ such that
\begin{equation}\label{eqn:zariski2}
 \langle\varphi^*\overline{D}_k\cdots\varphi^*\overline{D}_n\rangle\overline{M}^{\cdot k}\cdot\varphi^*\overline{D}_{n+1}\cdots\varphi^*\overline{D}_d+\varepsilon\geq\langle{\overline{D}'}^{\cdot k}\cdot\overline{D}_k\cdots\overline{D}_n\rangle\overline{D}_{n+1}\cdots\overline{D}_d.
\end{equation}
Since $\overline{D}'=\overline{P}+\overline{N}'$ and $\varphi^*\overline{D}'=\overline{M}+(\varphi^*\overline{D}'-\overline{M})$ are admissible approximations of $\overline{D}'$, there exists an admissible approximation $(\psi;\overline{Q})$ of $\overline{D}'$ such that $(\varphi;\overline{M})\leq(\psi;\overline{Q})$ and $(\varphi;\varphi^*\overline{P})\leq(\psi;\overline{Q})$.
Since $\psi^*\overline{P}\leq\overline{Q}$ and $\avol(\overline{P})=\avol(\overline{Q})=\avol(\overline{D}')$, we have $\psi^*\overline{P}=\overline{Q}$ by Corollary~\ref{cor:equalnefbig}.
Thus, by Lemma~\ref{lem:aint2} (3), we have
\begin{equation}\label{eqn:zariski3}
 \langle\overline{D}_k\cdots\overline{D}_n\rangle\overline{P}^{\cdot k}\cdot\overline{D}_{n+1}\cdots\overline{D}_d\geq\langle\varphi^*\overline{D}_k\cdots\varphi^*\overline{D}_n\rangle\overline{M}^{\cdot k}\cdot\varphi^*\overline{D}_{n+1}\cdots\varphi^*\overline{D}_d.
\end{equation}
By (\ref{eqn:zariski1}), (\ref{eqn:zariski2}), and (\ref{eqn:zariski3}), we have
\[
 \langle\overline{D}_k\cdots\overline{D}_n\rangle\overline{P}^{\cdot k}\cdot\overline{D}_{n+1}\cdots\overline{D}_d+2\varepsilon\geq\langle\overline{D}^{\cdot k}\cdot\overline{D}_k\cdots\overline{D}_n\rangle\overline{D}_{n+1}\cdots\overline{D}_d
\]
for every $\varepsilon>0$.
Hence we conclude the proof.
\end{proof}

\section*{Acknowledgments}

The author is greatly indebted to Professors Boucksom and Huayi Chen for stimulating conversations and for their papers \cite{Bou_Fav_Mat06} and \cite{Chen11}.
The author is grateful to Professor Maillot and the Institut de Math\'ematiques de Jussieu for their hospitality.
A part of this work was done in the institute.
The author is grateful to the referees for careful reading and useful suggestions.

\bibliography{ikoma}
\bibliographystyle{plain}

\end{document}